\documentclass[10pt,a4paper]{article}         % onecolumn (standard format)

\usepackage{color}
\usepackage{a4,amsmath,amssymb,amscd,latexsym,amsthm} 
\usepackage{bbm} 
\usepackage{epsfig,color} 
\usepackage{amsfonts}
\usepackage{mathrsfs}
\usepackage{overpic}
\usepackage{subcaption}
\usepackage{paralist}
\usepackage{graphicx}
\usepackage{trfsigns}
\usepackage{url}
\usepackage{stmaryrd}

\usepackage{tikz}
\usetikzlibrary{calc,trees,positioning,arrows,chains,shapes.geometric,%
    decorations.pathreplacing,decorations.pathmorphing,shapes,%
    matrix,shapes.symbols,backgrounds,shadows}
  
\usepackage{tikz-cd} 

\newtheoremstyle{mystyle}
{3pt}% Space above
{3pt}% Space below
{\itshape}% Body font
{}% Indent amount
{\bold}% Theorem head font
{.}% Punctuation after theorem head
{.5em}% Space after theorem head
{}% Theorem head spec (can be left empty, meaning ‘normal’ )

\newtheorem{definition}{Definition}[section]
\newtheorem{theorem}[definition]{Theorem}
\newtheorem{remark}[definition]{Remark}
\newtheorem{defthe}[definition]{Theorem and Definition}

\numberwithin{equation}{section}

\usepackage[]{algorithm2e}
\usepackage{algorithmicx}

\newcommand{\Gi}{\Gamma}
\newcommand{\Go}{\Gamma_\text{out}}

\newcommand{\Gb}{\mathrm{\Gamma_{bottom}}}
\newcommand{\Gl}{\mathrm{\Gamma_{left}}}
\newcommand{\Gr}{\mathrm{\Gamma_{right}}}
\newcommand{\Gt}{\mathrm{\Gamma_{top}}}

\def\nd#1{\frac{\partial #1}{\partial n}} 
\def\td#1{\frac{\partial #1}{\partial t}} 

\def\intdxdt#1{\int_{0}^{T} \int_{X(\Gi)} #1 \hspace{.5mm}dx\hspace{.5mm}dt}

\newcommand{\kathrin}[1]{{\color{black}#1}}

\newcommand{\rrevkathrin}[1]{{\color{black}#1}}
\newenvironment{rev} {\color{black}} {\color{black}}
\newenvironment{rrev} {\color{black}} {\color{black}}
\newenvironment{rrrev} {\color{black}} {\color{black}}

\title{\bf Suitable Spaces for Shape Optimization}

\author{Kathrin Welker\thanks{Helmut-Schmidt-University / University of the Federal Armed Forces Hamburg, Faculty of Mechanical Engineering, Holstenhofweg 85, 22043 Hamburg, Germany (\tt welker@hsu-hh.de)}} 
\date{}%\today

\begin{document}

\maketitle

\begin{abstract}
\noindent
The differential-geometric structure of the manifold of smooth shapes is applied to the theory of shape optimization problems. In particular, a Riemannian shape gradient with respect to the first Sobolev metric and the Steklov-Poincar\'{e} metric are defined. Moreover, the covariant derivative associated with the first Sobolev metric is deduced in this paper. The explicit expression of the covariant derivative leads to a definition of the Riemannian shape Hessian with respect to the first Sobolev metric.
In this paper, we give a brief overview of various optimization techniques based on the gradients and the Hessian. 
Since the space of smooth shapes limits the application of the optimization techniques, this paper extends the definition of smooth shapes to $H^{1/2}$-shapes, which arise naturally in shape optimization problems. 
We define a diffeological structure on the new space of $H^{1/2}$-shapes. This can be seen as a first step towards the formulation of optimization techniques on diffeological spaces.
\end{abstract}


%%%%%%%%%%%%%%%%%%%%%%%%%%%%%%%%%%%%%%%%%%%%%%%%%%%%%%%%%%%%%%%%%%%%%%%%%%%%%%%%%%%%%%%%%%%%%%%%%

\section{Introduction}

Shape optimization is of great importance in a wide range of applications.
A lot of real world problems can be reformulated as shape optimization problems which are constrained by partial differential equations (PDE). Aerodynamic shape optimization \cite{AIAA-2013}, acoustic shape optimization \cite{Berggren2016largescale}, optimization of interfaces in transmission problems \cite{Langer-2015,Paganini}, image restoration and segmentation \cite{HR-2003}, electrochemical machining \cite{HL-2011} and inverse modelling of skin structures \cite{Skin-2015} can be mentioned as examples. 
The subject of shape optimization is covered by several fundamental monographs, see, for instance, \cite{Delfour-Zolesio-2001,SokoZol}.

Questions like  \emph{``How can shapes be defined?''} or  \emph{``How does the set of all shapes look like?''} have been extensively studied in recent decades.
Already in 1984, David G.~Kendall has introduced the notion of a shape space in \cite{Kendall}. Often, a shape space is just modeled as a linear (vector) space, which in the simplest case is made up of vectors of landmark positions (cf.~\cite{CootesTaylorCooperGraham,Kendall}).
However, there is a large number of different shape concepts, e.g., plane curves \cite{MichorMumford,MioSrivastavaJoshi}, surfaces in higher dimensions \cite{BauerHarmsMichor,KMP,MichorMumford2}, boundary contours of objects \cite{FuchsJuettlerScherzerYang,LingJacobs,RumpfWirth2}, multiphase objects \cite{WirthRumpf}, characteristic functions of measurable sets \cite{Zolesio} and morphologies of images \cite{DroskeRumpf}.
In a lot of processes in engineering, medical imaging and science, there is a great interest to equip the space of all shapes with a significant metric to distinguish between different shape geometries.
In the simplest shape space case (landmark vectors), the distances between shapes can be measured by the Euclidean distance, but in general, the study of shapes and their similarities is a central problem.
In order to tackle natural questions like \emph{``How different are shapes?''}, \emph{``Can we determine the measure of their difference?''} or \emph{``Can we infer some information?''} mathematically, we have to put a metric on the shape space. There are various types of metrics on shape spaces, e.g.,
inner metrics \cite{BauerHarmsMichor,MichorMumford} like the Sobolev metrics, outer metrics \cite{BegLDDMM,Kendall,MichorMumford}, metamorphosis metrics \cite{Holm,TrYo}, the Wasserstein or Monge-Kantorovic metric on the shape space of probability measures \cite{AmbrosioGigliSava,BenamouBrenier}, the Weil-Petersson metric \cite{Kushnarev}, current metrics \cite{DurrlemanPennec} and metrics based on elastic deformations \cite{FuchsJuettlerScherzerYang,RumpfWirth2}.
However, it is a challenging task to model both, the shape space and the associated metric. There does not exist a common shape space or shape metric suitable for all applications. Different approaches lead to diverse models. The suitability of an approach depends on the requirements in a given situation.

In contrast to a finite dimensional optimization problem, which can be obtained, e.g., by representing shapes as splines, the connection of shape calculus with infinite dimensional spaces \cite{Delfour-Zolesio-2001,ItoKunisch,SokoZol} leads to a more flexible approach. 
In recent work, it has been shown that PDE constrained shape optimization problems can be embedded in the framework of optimization on shape spaces. E.g., in \cite{schulz2014structure}, shape optimization is considered as optimization on a Riemannian shape manifold, the manifold of smooth shapes.
Moreover, an inner product, which is called Steklov-Poincar\'{e} metric, for the application of finite element (FE)  methods is proposed in \cite{schulz2015Steklov}.

\kathrin{First, we concentrate on the particular manifold of smooth shapes and consider the first Sobolev and the Steklov-Poincar\'{e} metric in this paper. 
	\begin{rrev}
		The definition of the Riemannian shape gradient with respect to these two metrics results in\end{rrev}
	the formulation of gradient based optimization algorithms. One aim of this paper is to give an overview of the optimization techniques in the space of smooth shapes
	\begin{rrev}
		together with the first Sobolev and Steklov-Poincar\'{e} metric.
	\end{rrev}
	This paper extends the gradient based results in \cite{gsi2017}, where the theory of PDE constrained shape optimization problems is connected with the differential-geometric structure of the space of smooth shapes. % to a Riemannian shape Hessian. 
	To be more precisely, this paper \begin{rrev}aims at\end{rrev} the definition of a Riemannian shape Hessian with respect to the first Sobolev metric.
	In order to formulate such a definition, the covariant derivative needs to be specified. This paper formulates a theorem about the covariant derivative associated with the first Sobolev metric. This opens the door for formulating higher order methods in the space of smooth shapes.}

The manifold of smooth shapes contains shapes with infinitely differentiable boundaries, which limits the practical applicability.
For example, in the setting of PDE constrained shape optimization, one has to deal with polygonal shape representations from a computational point of view. This is because FE methods are usually used to discretize the models. 
In \cite{schulz2015Steklov}, not only an inner product, the Steklov-Poincar\'{e} metric, is given but also a suitable shape space for the application of FE methods is proposed. The combination of this particular shape space and its associated inner product is an essential step towards applying efficient FE solvers as outlined in \cite{SiebenbornWelker_skin}. 
However, so far, this shape space and its properties are not investigated.
From a theoretical point of view, it is necessary to clarify its structure. If we do not know the structure, there is no chance to get control over the space.
\kathrin{Thus, this paper aims at a generalization of smooth shapes to shapes which arise naturally in shape optimization problems. We define the space of so-called \emph{$H^{1/2}$-shapes}. Moreover, we clarify its structure as a diffeological one and, thus, go towards the formulation of optimization techniques on diffeological spaces.
	Since a diffeological space is one of the generalizations of manifolds, this paper formulates a theorem which  clarifies the difference between manifolds and diffeological spaces.}

This paper is organized as follows. 
In Section \ref{section_overview}, besides a short overview of basic concepts in shape optimization \kathrin{(Subsection~\ref{subsection_basicconcepts})}, the connection of shape calculus with the differential-geometric structure of shape spaces is stated \kathrin{(Subsection~\ref{subsection_shapecalculus})}. \kathrin{In particular, the Riemannian shape gradients with respect to the first Sobolev and Steklov-Poincar\'{e} metric are defined and the Riemannian shape Hessian with respect to the first Sobolev metric is given. One of the main theorems of this paper is Theorem~\ref{theorem_covdervH1}, which specifies the covariant derivative associated with the first Sobolev metric \begin{rrev}necessary for the definition of the Riemannian shape Hessian with respect to the first Sobolev metric.\end{rrev} 
	\begin{rrev}
		Thanks to the definition of the Riemannian shape Hessian we are able to formulate the Newton method in the space of smooth shapes together with the first Sobolev metric.
	\end{rrev}
	\begin{rrev}Additionally,\end{rrev} we give a brief overview of \begin{rrev}first order\end{rrev} optimization techniques based on gradients \begin{rrev}with respect to the first Sobolev metric as well as the Steklov-Poincar\'{e} metric\end{rrev}. In particular, \begin{rrev}Subsection~\ref{subsection_shapecalculus} ends with a comparison of the gradient based algorithms for a specific example.\end{rrev}}
Section \ref{section_shapespace} is concerned with the space of $H^{1/2}$-shapes.
\kathrin{First, we give a brief introduction in diffeological spaces and explain the difference between these spaces and manifolds (Subsection~\ref{subsection_definitions_diffspace_B12}). 
	The first main theorem of Section \ref{section_shapespace} is Theorem~\ref{theorem_diffman}, which specifies the difference between diffeological spaces and manifolds.
	In Subsection~\ref{subsection_shapespaceB12}, the space of $H^{1/2}$-shapes is defined.
	Here, Theorem~\ref{Theorem:DiffStructure}, which is the third and last of the main theorems in this paper, endows the space of $H^{1/2}$-shapes with its diffeological structure.}

%%%%%%%%%%%%%%%%%%%%%%%%%%%%%%%%%%%%%%%%%%%%%%%%%%%%%%%%%%%%%%%%%%%%%%%%%%%%%%%%%%%%%%%%%%%%%%%%%

\section{Optimization in shape spaces}\label{section_overview} 

First, we set up notation and terminology of basic shape optimization concepts (Subsection~\ref{subsection_basicconcepts}). 
Afterwards, shape calculus is combined with geometric concepts of shape spaces (Subsection~\ref{subsection_shapecalculus}). \kathrin{In \cite{gsi2017}, the theory of shape optimization problems constrained by partial differential equations is already connected with the differential-geometric structure of the space of smooth shapes. Moreover, gradient-based methods are outlined. However, Subsection~\ref{subsection_shapecalculus} extends these results to a Riemannian shape Hessian, for which the covariant derivative needs to be specified. This opens the door for formulating higher order methods in space of smooth shapes. \begin{rrev}
		In particular, we formulate a Newton method on the space of smooth shapes based on the definition of the Riemannian shape Hessian.
	\end{rrev}
}

\subsection{Basic concepts in shape optimization}
\label{subsection_basicconcepts}

\begin{rrev}
	This section sets up notation and terminology of basic shape optimization concepts used in this paper.
	For a detailed introduction into shape calculus, we refer to the monographs \cite{Delfour-Zolesio-2001,SokoZol}. 
\end{rrev}

One of the main focuses of shape optimization is to investigate shape functionals and solve shape optimization problems.
First, we give the definition of a shape functional.

\begin{definition}[Shape functional]
	\label{def_shapefunctional}
	Let $D$ denote a non-empty subset of $\mathbb{R}^d$, where $d\in\mathbb{N}$. Moreover, $\mathcal{A}\subset \{\Omega\colon \Omega \subset D\}$ denotes a set of subsets. A function
	$$J\colon \mathcal{A}\to \mathbb{R}\text{, } \Omega\mapsto J(\Omega)$$
	is called a shape functional.
\end{definition}

Let $J\colon \mathcal{A}\to \mathbb{R}$ be a shape functional, where $\mathcal{A}$ is a set of subsets $\Omega$ as in Definition \ref{def_shapefunctional}.
An \emph{unconstrained shape optimization problem} is given by
\begin{equation}
\label{minproblem}
\min_{\Omega\in \mathcal{A}} J(\Omega).
\end{equation}
Often, shape optimization problems are constrained by equations, e.g., equations involving an unknown function of two or more variables and at least one partial derivative of this function. 
\kathrin{In this case, the objective functional $J$ has two arguments, the shape $\Omega$ as well as the so-called \emph{state variable} $y$, where the state variable is the solution of the underlying constraint. A \emph{constrained shape optimization problem} reads as
	\begin{align}
	\min_{(\Omega,y)\in \mathcal{A}\times \mathcal{X}(\Omega)} &J(\Omega,y)\label{minproblem2}\\
	\text{s.t. } \hspace{.3cm}\quad &y=y(\Omega) \text{ solves } \begin{rrev}
	\mathcal{F}
	\end{rrev}(\Omega,y(\Omega))=0,
	\end{align}
	where $\mathcal{X}(\Omega)$ is usually a function space and the constraint $\begin{rrev}
	\mathcal{F}
	\end{rrev}(\Omega,y(\Omega))=0$ is given for example by a PDE or a system of PDEs.
}
When $J$ in (\ref{minproblem2}) depends on a solution of a PDE, we call the shape optimization problem \emph{PDE constrained}.

Let $D$ be as in \rrevkathrin{Definition \ref{def_shapefunctional}}. Moreover, let $\{F_t\}_{t\in[0,T]}$ be a family of mappings $F_t\colon \overline{D}\to\mathbb{R}^d$ such that $F_0=\kathrin{\text{id}}$, where $\overline{D}$ denotes the closure of $D$ and $T>0$.
This family transforms the domain $\Omega$ into new \emph{perturbed domains} 
\begin{equation*}
\Omega_t := F_t(\Omega)=\{F_t(x)\colon x\in \Omega\}\text{ with }\Omega_0=\Omega
\end{equation*}
and the boundary $\Gamma$ of $\Omega$ into new \emph{perturbed boundaries}
\begin{equation*}
\Gamma_t := F_t(\Gamma)=\{F_t(x)\colon x\in \Gamma\}\text{ with }\Gamma_0=\Gamma.
\end{equation*}
Such a transformation can be described by the \emph{velocity method} or by the \emph{perturbation of identity}.
We concentrate on the perturbation of identity, which is defined by $F_t(x):= x+tV(x)$, where $V$ denotes a sufficiently smooth vector field.

To solve shape optimization problems, we need their shape derivatives. 

\begin{definition}[Shape derivative]
	\label{def_shapeder}
	\kathrin{Let $D\subset \mathbb{R}^d$ be open, $\Omega\subset D$ and $k\in\mathbb{N}\cup \{\infty\}$. Moreover, let $\mathcal{C}^k_0(D,\mathbb{R}^d)$ denote the set of $\mathcal{C}^k(D,\mathbb{R}^d)$-functions which vanish on $\partial\Omega$.}
	The Eulerian derivative of a shape functional $J$ at $\Omega$ in direction $V\in\mathcal{C}^k_0(D,\mathbb{R}^d)$ is defined by
	\begin{equation}
	\label{eulerian}
	DJ(\Omega)[V]:= \lim\limits_{t\to 0^+}\frac{J(\Omega_t)-J(\Omega)}{t}. 
	\end{equation}
	If for all directions $V\in\mathcal{C}^k_0(D,\mathbb{R}^d)$ the Eulerian derivative (\ref{eulerian}) exists and the mapping 
	\begin{equation*}
	G(\Omega)\colon \mathcal{C}^k_0(D,\mathbb{R}^d)\to \mathbb{R}, \ V\mapsto DJ(\Omega)[V]
	\end{equation*}
	is linear and continuous, the expression $DJ(\Omega)[V]$ is called the shape derivative of $J$ at $\Omega$ in direction $V\in\mathcal{C}^k_0(D,\mathbb{R}^d)$. In this case, $J$ is called shape differentiable of class $\mathcal{C}^k$ at $\Omega$.
\end{definition}

\begin{rrev}
	\begin{remark}
		There are \begin{rrev}many\end{rrev} options to prove shape differentiability of shape functionals which depend on a solution of a PDE and to derive the shape derivative of a shape optimization problem. The min-max approach \cite{Delfour-Zolesio-2001}, the chain rule approach \cite{SokoZol}, the Lagrange method of C\'{e}a \cite{Cea-RAIRO} and the rearrangement method \cite{Ito-Kunisch-Peichl} have to be mentioned in this context.
		A nice overview about these approaches is given in \cite{Sturm}. 
	\end{remark}
\end{rrev}

The Hadamard Structure Theorem (cf.~\cite[Theorem 2.27]{SokoZol}) states that \kathrin{under certain assumptions the shape derivative is a distribution acting on the normal part of the perturbation field on the boundary}. 

\begin{rev}
	\begin{theorem}[Hadamard Structure Theorem]
		\label{theorem_HadamardStructure}
		Let $D$ and $\Omega$ be as in Definition~\ref{def_shapeder}. Moreover, let the shape functional $J$ be shape differentiable of class $\mathcal{C}^k$ at every domain $\Omega\subset D$ with $\mathcal{C}^{k-1}$-boundary $\Gamma=\partial\Omega$.
		Then there exists a scalar distribution $r\in \mathcal{C}^k_0(\Gamma)'$ such that 
		$G(\Omega)\in\mathcal{C}^k_0(\Omega,\mathbb{R}^d)'$ of $J$ at $\Omega$ is given by
		\begin{equation}
		\label{hadamard}
		G(\Omega)=\gamma_\Gamma'(r\cdot n).
		\end{equation}
		Here $ \mathcal{C}^k_0(\Gamma)'$ and $\mathcal{C}^k_0(\Omega,\mathbb{R}^d)'$ denote the dual spaces of $ \mathcal{C}^k_0(\Gamma)$ and $\mathcal{C}^k_0(\Omega,\mathbb{R}^d)$. Moreover, $$\gamma_\Gamma\colon \mathcal{C}^k_0(\overline{D},\mathbb{R}^d)\to \mathcal{C}^k_0(\Gamma,\mathbb{R}^d), \ v\to v\,\rule[-2mm]{.1mm}{4mm}_{\hspace{.6mm}\Gamma} $$ denotes the trace operator and $\gamma_\Gamma'$ its adjoint operator.
	\end{theorem} 
	\noindent
	Note that the Hadamard Structure Theorem \ref{theorem_HadamardStructure} actually states the existence of a scalar distribution $r=r(\Omega)$ on the boundary $\Gamma$ of a domain $\Omega$. However, in this paper, we always assume that $r$ is an integrable function. In general, if $r\in L^1(\Gamma)$, then $r$ is obtained in the form of the trace on $\Gamma$ of an element $G\in W^{1,1}(\Omega)$. This means that it follows from (\ref{hadamard}) that the shape derivative can be expressed more conveniently as
	\vspace{.2cm}
	\begin{equation*}
	DJ(\Omega)[V] =\int_\Gamma r\left<V, n\right> ds.
	\end{equation*}
\end{rev}

\begin{rev}
	If the objective functional is given by an integral over the whole domain, the shape derivative can be expressed as an integral over the domain, the so-called \emph{volume} or \emph{weak formulation}, and an integral over the boundary, the so-called \emph{surface} or \emph{strong formulation}:\end{rev}
\begin{align}
\label{dom_form}
DJ_\Omega[V]&:= \int_\Omega RV(x)\, dx  &\text{(volume / weak formulation)}\\[5pt]
\label{bound_form}
DJ_\Gamma[V]&:=\int_\Gamma r(s)\left<V(s), n(s)\right> ds  &\text{(surface / strong formulation)}
\end{align}
Here $r\in L^1(\Gamma)$ and $R$ is a differential operator acting linearly on the vector field $V$ with
$DJ_\Omega[V]=DJ(\Omega)[V]=DJ_\Gamma[V]$.
Recent \rrevkathrin{advances} in PDE constrained optimization on shape manifolds are based on the surface formulation, also called \emph{Hadamard-form}, as well as intrinsic shape metrics. Major effort in shape calculus has been devoted towards such surface expressions (cf.~\cite{Delfour-Zolesio-2001,SokoZol}), which are often very tedious to derive.
\kathrin{When one derives a shape derivative of an objective functional which is given by an integral over the domain, one first get the volume formulation. This volume form can be converted into its surface form by applying the integration by parts formula. In order to apply this formula, one needs a higher regularity of the state and adjoint of the underlying PDE. }
Recently, it has been shown that the \kathrin{weak} formulation has numerical advantages, see, for instance, \cite{Berggren,Langer-2015,HipPag_2015,Paganini}. In \cite{LaurainSturm2013}, also practical advantages of volume shape formulations have been demonstrated. 
However, volume integral forms of shape derivatives require an outer metric on the domain surrounding the shape boundary. 
\kathrin{In contrast to inner metrics, which can be seen as describing a deformable material that the shape itself is made of, the differential operator governing outer metrics is defined even outside of the shape (cf., e.g.,  \cite{BegLDDMM,Bookstein,Kendall,MichorMumford}). }
In \cite{schulz2015Steklov}, both points of view are harmonized by deriving a metric from an outer metric. 
Based on this metric, efficient shape optimization algorithms, which also reduce the analytical effort so far involved in the derivation of shape derivatives, are proposed in \cite{schulz2015Steklov,SiebenbornWelker_skin,Welker}. The next subsection \begin{rrev}
	concentrates on the question\end{rrev} how shape calculus and in particular shape derivatives can be combined with geometric concepts of shape spaces. This combination results in efficient optimization techniques in shape spaces.

\subsection{Shape calculus combined with geometric concepts of shape spaces}
\label{subsection_shapecalculus}

As pointed out in \cite{Schulz}, shape optimization can be viewed as optimization on Riemannian shape manifolds and the resulting optimization methods can be constructed and analyzed within this framework. This combines algorithmic ideas from \cite{Absil} with the Riemannian geometrical point of view established in \cite{BauerHarmsMichor}.
In this subsection, we analyze the connection of Riemannian geometry on the space of smooth shapes to shape optimization 
\begin{rrev}
	and extend the results in \cite{gsi2017}, which is also concerned with this connection, to a Riemannian shape Hessian and to a second order optimization method. In particular, we specify the covariant derivative associated with the first Sobolev metric on the space of smooth shapes (cf.~Theorem~\ref{theorem_covdervH1} in Subsection~\ref{Subsubsection:Be}), which results in the definition of the Riemannian shape Hessian with respect to this metric (cf.~Definition~\ref{definition_Hessian} in Subsection~\ref{Subsubesction:OptimizationBe}). The formulation of the covariant derivative and the shape Hessian opens the door for formulating higher order methods in space of smooth shapes (cf.~Algorithm~\ref{Algo:Newton} in Subsection~\ref{Subsubesction:OptimizationBe}).
\end{rrev}

\subsubsection{The space of smooth shapes}
\label{Subsubsection:Be}

\begin{rrev}
	We first introduce the space of smooth shapes and summarize some of its properties which are relevant for this paper from the literature \cite{BauerBruverisMICCAI,BauerHarmsMichor,BauerHarmsMichor_SobolevII,KrieglMichor,MichorMumford1,MichorMumford}. 
\end{rrev}
First, we concentrate on one-dimensional shapes, which
are defined as the images of simple closed smooth curves in the plane of the unit circle. Such simple closed smooth curves can be represented by embeddings from the circle $S^1$ into the plane $\mathbb{R}^2$, see, for instance, \cite{Kuehnel}. Therefore, the set of all embeddings from $S^1$ into $\mathbb{R}^2$, denoted by $\mathrm{Emb}(S^1,\mathbb{R}^2)$, represents all simple closed smooth curves in $\mathbb{R}^2$.
However, note that we are only interested in the shape itself and that images are not changed by re-parametrizations. Thus, all simple closed smooth curves which differ only by re-parametrizations can be considered equal to each other because they lead to the same image. Let $\mathrm{Diff}(S^1)$ denote the set of all diffeomorphisms from $S^1$ into itself. This set is a regular Lie group (cf.~\cite[Chapter~VIII, 38.4]{KrieglMichor}) and consists of all the smooth re-parametrizations mentioned above.
In \cite{MichorMumford1}, the \emph{set of all one-dimensional shapes} is characterized by
\begin{equation}
\label{B_e_2dim}
B_e(S^1,\mathbb{R}^2):= \mathrm{Emb}(S^1,\mathbb{R}^2)/\mathrm{Diff}(S^1),
\end{equation}  
i.e., the obit space of $\mathrm{Emb}(S^1,\mathbb{R}^2)$ under the action by composition from the right by the Lie group $\mathrm{Diff}(S^1)$.
A particular point on $B_e(S^1,\mathbb{R}^2)$ is represented by a curve 
$
c\colon S^1\to \mathbb{R}^2 , \ \theta\mapsto c(\theta)
$
and illustrated in the left picture of Figure \ref{figure_Be}.
The tangent space is isomorphic to the set of all smooth normal vector fields along $c$, i.e.,
\begin{equation}
\label{isomorphismTcBe}
T_cB_e(S^1,\mathbb{R}^2)\cong\left\{h\colon h=\alpha n,\, \alpha\in \mathcal{C}^\infty(S^1)\right\},
\end{equation}
where $n$ denotes the exterior unit normal field to the shape boundary $c$ such that $n (\theta)\perp c_\theta(\theta)$ for all $\theta\in S^1$, where $c_\theta=\frac{\partial c}{\partial\theta}$ denotes the circumferential derivative as in \cite{MichorMumford1}. 
Since we are dealing with parametrized curves, we have to work with the arc length and its derivative. Therefore, we use the following notation:
\begin{align}
ds & = \lvert c_\theta \rvert d\theta & \text{(arc length \kathrin{with respect to $c$})}\label{arclength} \\
D_s & = \frac{\partial_\theta}{\lvert c_\theta \rvert} & \text{(arc length derivative \kathrin{with respect to $c$})}\label{Def:arclength}
\end{align}

\kathrin{
	\begin{remark}
		Some properties of the operator $D_s$ can be found in, e.g., \cite{MichorMumford}. In  \cite{BauerHarmsMichor}, this operator is considered for higher dimensions and its connection with the Bochner-Laplacian is given. 
	\end{remark}
}

In \cite{KrieglMichor}, it is proven that the shape space $B_e(S^1,\mathbb{R}^2)$ is a smooth manifold. \emph{Is it even perhaps a Riemannian shape manifold?} This question was investigated by Peter W.~Michor and David Mumford. They show in \cite{MichorMumford1} that the standard $L^2$-metric on the tangent space is too weak because it induces geodesic distance equals zero. This phenomenon is called the \emph{vanishing geodesic distance phenomenon}. The authors employ a curvature weighted $L^2$-metric as a remedy and prove that the vanishing phenomenon does not occur for this metric. Several Riemannian metrics on this shape space are examined in further publications, e.g., \cite{BauerHarmsMichor,MichorMumford2,MichorMumford}. All these metrics arise from the $L^2$-metric by putting weights, derivatives or both in it. In this manner, we get three groups of metrics: the \emph{almost local metrics} which arise by putting weights in the $L^2$-metric (cf.~\cite{BauerHarmsMichor_SobolevII,MichorMumford}), the \emph{Sobolev metrics} which arise by putting derivatives in the $L^2$-metric (cf.~\cite{BauerHarmsMichor,MichorMumford}) and the \emph{weighted Sobolev metrics} which arise by putting both, weights and derivatives, in the $L^2$-metric (cf.~\cite{BauerHarmsMichor_SobolevII}).
It can be shown that all these metrics do not induce the phenomenon of vanishing geodesic distance under special assumptions. To list all these goes beyond the scope of this paper, but they can be found in the above-mentioned publications. All Riemannian metrics mentioned above are \emph{inner metrics}. \kathrin{As already mentioned above, this means that  
	the deformation is prescribed on the shape itself and the ambient space\footnote{\begin{rrev}The \emph{ambient space} of a mathematical object is the space surrounding that mathematical object along with the object itself.
	\end{rrev}} stays fixed.}

In the following, we clarify  \kathrin{briefly}  how the above-mentioned inner Riemannian metrics can be defined on the shape space $B_e(S^1,\mathbb{R}^2)$. 
\kathrin{For details we refer to \cite{MichorMumford}. Moreover, we refer to \cite{BauerBruverisMICCAI} for a comparison of  an inner metric on $B_e(S^1,\mathbb{R}^2)$ with the diffeomorphic matching framework which works with outer metrics.}

\kathrin{First, we define a Riemannian metric on the space $\mathrm{Emb}(S^1,\mathbb{R}^2)$, which is} 
a family $g=\left(g_c(h,k)\right)_{c\in \mathrm{Emb}(S^1,\mathbb{R}^2)}$ of inner products $g_c(h,k)$, where $h$ and $k$ denote vector fields along $c\in\mathrm{Emb}(S^1,\mathbb{R}^2)$. The most simple inner product on the tangent bundle to $\text{Emb}(S^1,\mathbb{R}^2)$ is the standard $L^2$-inner product $g_c(h,k) := \int_{S^1}\left< h,k\right> ds$.
Note that 
\begin{equation}
\label{TangentSpace}
T_c\text{Emb}(S^1,\mathbb{R}^2)\cong \mathcal{C}^\infty(S^1,\mathbb{R}^2) \qquad \forall\, c\in \text{Emb}(S^1,\mathbb{R}^2)
\end{equation}
and that a tangent vector $h\in T_c\text{Emb}(S^1,\mathbb{R}^2)$ has an orthonormal decomposition into smooth tangential components $h^\top$ and normal components $h^\perp$ (cf.~\cite[Section~3, 3.2]{MichorMumford1}).
In particular, $h^\perp$ is an element of the bundle of tangent vectors which are normal to the $\text{Diff}(S^1)$-orbits denoted by $\mathcal{N}_c$. This normal bundle is well defined and is a smooth vector subbundle of the tangent bundle.
In \cite{MichorMumford1}, it is outlined how the restriction of the metric $g_c$ to the subbundle $\mathcal{N}_c$ gives the quotient metric. 
The quotient metric induced by the $L^2$-metric is given by
\begin{equation}
\begin{split}
g^0\colon T_cB_e(S^1,\mathbb{R}^2)\times T_cB_e(S^1,\mathbb{R}^2) & \to \mathbb{R}\text{, } \\
(h,k) & \mapsto \int_{S^1}\left<\alpha,\beta \right>  ds, 
\end{split}
\end{equation}
where $h=\alpha n$ and $k=\beta n$ denote two elements of the tangent space $T_cB_e(S^1,\mathbb{R}^2)$ given in (\ref{isomorphismTcBe}). Unfortunately, in \cite{MichorMumford1}, it is shown that this $L^2$-metric induces vanishing geodesic distance, as already mentioned above.

For the following discussion, among all the above-mentioned Riemannian metrics, we pick the first Sobolev metric \begin{rrev}which does not induce the phenomenon of vanishing geodesic distance\end{rrev} \kathrin{(cf.~\cite{MichorMumford})}.  \begin{rrev}On $B_e(S^1,\mathbb{R}^2)$ it is defined as follows:\end{rrev}

\begin{definition}[First Sobolev metric on $B_e(S^1,\mathbb{R}^2)$]
	The first Sobolev metric on $B_e(S^1,\mathbb{R}^2)$ is given by
	\begin{equation}
	\label{Sobolev-metric_g1}
	\begin{split}
	g^1\colon T_cB_e(S^1,\mathbb{R}^2)\times T_cB_e(S^1,\mathbb{R}^2) & \to \mathbb{R}\text{, } \\
	(h,k) & \mapsto \int_{S^1}\left< \left(I - A \kathrin{D^2_s}\right)h,k\right> ds, 
	\end{split}
	\end{equation}
	where 
	$A>0$ and \kathrin{$D_s$ denotes the arc length derivative with respect to $c$ defined in (\ref{Def:arclength})}. 
\end{definition}

An essential operation in Riemannian geometry is the covariant derivative. In differential geometry, it is often written in terms of the Christoffel symbols. In \cite{BauerHarmsMichor}, Christoffel symbols associated with the Sobolev metrics are provided. However, in order to provide a relation with shape calculus,
another representation of the covariant derivative in terms of the Sobolev metric $g^1$ is needed. 
\kathrin{Now, we get to the first main theorem of this paper.}
The Riemannian connection provided by \kathrin{this} theorem makes it possible to specify the Riemannian shape Hessian.

\begin{theorem}
	\label{theorem_covdervH1}
	Let $A>0$ and let $h,m\in T_c\text{\emph{Emb}}(S^1,\mathbb{R}^2)$ denote vector fields along $c\in\text{\emph{Emb}}(S^1,\mathbb{R}^2)$. \kathrin{The arc length derivative with respect to $c$ is denoted by $D_s$ as in (\ref{Def:arclength}).} Moreover, $L_1:= I-AD_s^2$ is a differential operator on $\mathcal{C}^\infty(S^1,\mathbb{R}^2)$ and $L_1^{-1}$ denotes its inverse operator. The covariant derivative associated with the Sobolev metric $g^1$ can be expressed as
	\begin{equation}
	\label{cov_der}
	\nabla_m h=L_1^{-1}(K_1(h))\text{ with }K_1 := \frac{1}{2}\left<D_s m,v\right>\left(I+AD_s^2\right),
	\end{equation}
	where $v=\frac{c_\theta}{|c_\theta|}$ denotes the unit tangent vector.
\end{theorem}

\begin{remark}
	The inverse operator $L_1^{-1}$ in Theorem~\ref{theorem_covdervH1} is an integral operator whose kernel has an expression in terms of the arc length distance between two points on a curve and their unit normal vectors (cf.~\cite{MichorMumford}). For the existence and more details about $L_1^{-1}$ we refer to \cite{MichorMumford}. 
\end{remark}

\emph{Proof of Theorem~\ref{theorem_covdervH1}.}
Let $h,k,m$ be vector fields on $\mathbb{R}^2$ along $c\in\text{Emb}(S^1,\mathbb{R}^2)$.
Moreover, $d(\cdot)[m]$ denotes the directional derivative in direction $m$.
\begin{rrev}
	From \cite[Subsection 4.2, formula (3)]{MichorMumford}, we have
	\begin{equation}
	\label{DerivativeL}
	d(L(h))[m]=A\left<D_sm,v\right>D_s^2h+AD_s\left<D_sm,v\right>D_sh
	\end{equation}
	Applying (\ref{DerivativeL}), we obtain in analogy to the computations in  \cite[Subsection 4.2]{MichorMumford}
\end{rrev}
\begin{equation}
\label{Dg1}
\begin{split}
& d\left(g_c^1(h,k)\right)[m] =d\left(\int_{S^1}\left<L_1(h),k\right>ds\right)[m] \\
&=\int_{S^1}\left<d\left(L_1(h)\right)[m],k\right>ds+\int_{S^1}\left<L_1(h),k\right>\begin{rrev}
\left<D_s m,v\right>ds 
\end{rrev} \\
&\hspace*{-2mm}\begin{rrev}\stackrel{(\ref{DerivativeL})}{=}\int_{S^1}\left<A\left<D_sm,v\right>D_s^2h+AD_s\left<D_sm,v\right>D_sh,k\right>ds\end{rrev}\\
&\begin{rrev}\phantom{\int_{S^1}}+\int_{S^1}\left<L_1(h),k\right>\left<D_s m,v\right>ds.\end{rrev} %\\
\end{split}
\end{equation}
Since the differential operator $D_s$ is anti self-adjoint for the $L^2$-metric $g^0$, i.e., 
\rrevkathrin{
	\begin{equation}
	\label{selfadjoint}
	\int_{S^1} \left<D_sh,k\right>ds=\int_{S^1} \left<h,-D_sk\right>ds,
	\end{equation}
}
\begin{rrev}
	we get from (\ref{Dg1})
	\begin{equation}
	\label{Dg2}
	\begin{split}
	d\left(g_c^1(h,k)\right)[m] &=\int_{S^1}2A\left<D_s m,v\right>\left<D_s^2h,k\right>ds+\int_{S^1}\left<h,k\right>\left<D_s m,v\right>ds\\
	&\hspace*{.6cm}-\int_{S^1}A\left<D_s^2h,k\right>\left<D_s m,v\right>ds\\
	&=\int_{S^1}\left<D_s m,v\right>\left(\left<h,k\right>+A\left<D_s^2h,k\right>\right)ds  .
	\end{split}
	\end{equation}
\end{rrev}

Now, we proceed analogously to the proof of Theorem~2.1 in \cite{Schulz}, which exploits the product rule for Riemannian connections. Thus, we conclude from
\begin{align}
& d\left(g_c^1(h,k)\right)[m]\nonumber\\
& \stackrel{(\ref{Dg2})}=\int_{S^1}\left<D_s m,v\right>\left[\frac{1}{2}\left(\left<h,k\right>+A\left<D_s^2h,k\right>\right)+\frac{1}{2}\left(\left<h,k\right>+A\left<D_s^2h,k\right>\right)\right]ds\nonumber\\
&\stackrel{\rrevkathrin{(\ref{selfadjoint})}}=\int_{S^1}\left<\frac{1}{2}\left<D_s m,v\right>\left(I+AD_s^2\right)h,k\right>+\left<h,\frac{1}{2}\left<D_s m,v\right>\left(I+AD_s^2\right)k\right>ds\nonumber\\
&\hspace*{2.3mm}=\int_{S^1}\left<L_1\left[L_1^{-1}\left(\frac{1}{2}\left<D_s m,v\right>\left(I+AD_s^2\right)h\right)\right],k\right>ds\nonumber\\
&\hspace*{.6cm}+\int_{S^1}\left<h,L_1\left[L_1^{-1}\left(\frac{1}{2}\left<D_s m,v\right>\left(I+AD_s^2\right)k\right)\right]\right>ds\nonumber\\
&\hspace*{2.3mm}=g_c^1\left(L_1^{-1}\left(\frac{1}{2}\left<D_s m,v\right>\left(I+AD_s^2\right)h\right),k\right)\nonumber\\
&\hspace*{.6cm}+g_c^1\left(h,L_1^{-1}\left(\frac{1}{2}\left<D_s m,v\right>\left(I+AD_s^2\right)k\right)\right)\nonumber
\end{align}
that the covariant derivative associated with $g^1$ is given by (\ref{cov_der}).\qed

\begin{rrev}
	\begin{remark}
		For the sake of completeness it should be mentioned that the shape space $B_e(S^1,\mathbb{R}^2)$ and its theoretical results can be generalized to higher dimensions.
		Let $M$ be a compact manifold and let $N$ denote a Riemannian manifold with $\text{\emph{dim}}(M)<\text{\emph{dim}}(N)$. In \cite{MichorMumford2}, the \emph{space of all submanifolds} of type $M$ in $N$ is defined by
		\begin{equation}
		B_e(M,N):=\text{\emph{Emb}}(M,N)/\text{\emph{Diff}}(M).
		\end{equation}
		In Figure \ref{figure_Be}, the left picture illustrates a two-dimensional shape which is an element of the shape space $B_e(S^2,\mathbb{R}^3)$. In contrast, the right shape in this figure is a two-dimensional shape which is not an element of this shape space. Note that the vanishing geodesic distance phenomenon occurs also for the $L^2$-metric in higher dimensions as verified in \cite{MichorMumford2}. For the definition of the Sobolev metric $g^1$ in higher dimensions we refer to \cite{BauerHarmsMichor}.
	\end{remark}
\end{rrev}

\begin{figure}
	\begin{minipage}{.5\textwidth}
		\vspace*{-8mm}
		\hspace{-.7cm}
		\includegraphics[width=.63\linewidth]{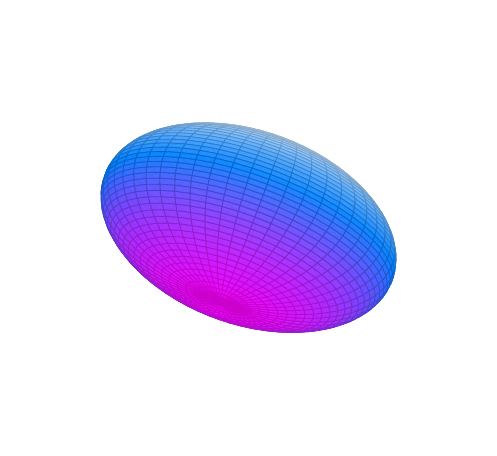}
		\hspace{-1.7cm}
		\vspace{-1cm}
		\includegraphics[width=.71\linewidth]{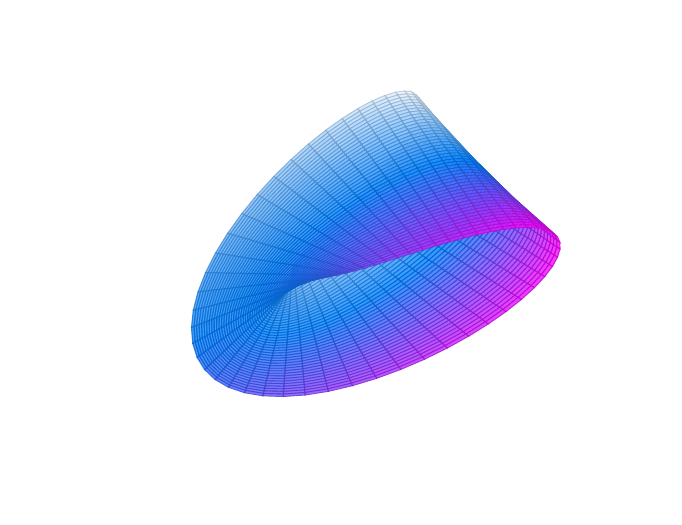}\\[9pt]
		Example of an element of $B_e(S^2,\mathbb{R}^3)$ (left) and a shape which is not an element of $B_e(S^2,\mathbb{R}^3)$ (right).
	\end{minipage}
	\hspace{.3cm}
	\begin{minipage}{.43\textwidth}
		\begin{center}
			\includegraphics[width=0.35\linewidth]{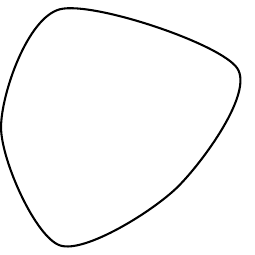}\hspace{.7cm}
			\includegraphics[width=0.35\linewidth]{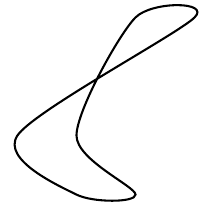}\\[5pt]
		\end{center}
		Example of an embedding (left) and an immersion which is not an embedding (right).
	\end{minipage}
	\caption{Examples of one- and two-dimensional shapes.}
	\label{figure_Be}
\end{figure}

\subsubsection{Optimization in the space of smooth shapes}
\label{Subsubesction:OptimizationBe}

\begin{rrev}
	In the following, we focus on two Riemannian metrics  on the space of smooth shapes $B_e$, the first Sobolev metric $g^1$ introduced in Subsection~\ref{Subsubsection:Be} and the  Steklov-Poincar\'{e} metric $g^S$ defined below.
	The aim of this subsection is to provide some optimization techniques in $(B_e,g^1)$ and $(B_e,g^S)$.  
	
	The subsection is structured in three paragraphs.
	The first paragraph considers the first Sobolev metric $g^1$, where we repeat firstly some relevant results from \cite{gsi2017}. Afterwards, we built on our findings of the previous subsection and extend the results in \cite{gsi2017}. More precisely, thanks to the specification of the covariant derivative associated with $g^1$ in the previous subsection, we are able to define  the Riemannian shape Hessian with respect to $g^1$ and formulate the Newton method in $(B_e,g^1)$, which is based on this definition. 
	As we will see below, if we consider Sobolev metrics, we have to deal with surface formulations of shape derivatives. 
	An intermediate and equivalent result in the process of deriving these expressions is the volume expression as already mentioned above.
	These volume expressions are preferable over surface forms.
	This is not only because of saving analytical effort, but also due to additional regularity assumptions, which usually have to be required in order to transform volume into surface forms, as well as because of saving programming effort.
	However, in the case of the more attractive volume formulation, the shape manifold $B_e$ and the corresponding inner products $g^1$ are not appropriate. 
	One possible approach to use volume forms is addressed in the second paragraph of this subsection, which considers Steklov-Poincar\'{e} metrics. 
	We summarize some of the main results related to this metric from \cite{schulz2015Steklov} with view on optimization methods.
	Finally, the third paragraph of this subsection considers a specific example and concludes this subsection with a brief discussion about the two approaches resulting from considering the first Sobolev and the Steklov-Poincar\'{e} metric. Since this paper does not focus on numerical investigations, we pick an example which is already implemented in \cite{schulz2015Steklov,Welker} to illustrate the main differences between the two approaches.
	
\end{rrev}

\paragraph{Optimization based on first Sobolev metrics.}
\label{SubsectionSob}

We consider the Sobolev metric $g^1$ on the shape space $B_e$. \kathrin{In particular, this means that we consider elements of $B_e$, i.e., smooth boundaries $\Gamma$ of the domain $\Omega$ under consideration in the following.} The Riemannian connection with respect to this metric, which is given in Theorem~\ref{theorem_covdervH1}, makes it possible to specify the Riemannian shape Hessian of an optimization problem. 

First, we detail the \emph{Riemannian shape gradient} \begin{rrev}
	from  \cite{gsi2017}.
\end{rrev}
Due to the Hadamard Structure Theorem, there exists a scalar distribution $r$ on the boundary $\Gamma$ of the domain $\Omega$ under consideration. If we assume $r\in L^1(\Gamma)$, the shape derivative can be expressed on the boundary $\Gamma$ of $\Omega$ (cf.~(\ref{bound_form})). 
The distribution $r$ is often called the \emph{shape gradient} in the literature. 
However, note that gradients depend always on chosen scalar products defined on the space under consideration. Thus, it rather means that $r$ \rrevkathrin{should be} the usual $L^2$-shape gradient, i.e., \begin{rrev}the gradient with respect to the $L^2$-scalar product\end{rrev}. If we want to optimize on a shape manifold, we have to find a representation of the shape gradient with respect to a Riemannian metric defined on the shape manifold under consideration. This representation is called the Riemannian shape gradient.
The shape derivative can be expressed more concisely as
\begin{equation}
\label{concisely_expression}
DJ_\Gamma[V]=\int_\Gamma \alpha r \ ds
\end{equation}
if $V\,\rule[-2mm]{.1mm}{4mm}_{\hspace{.6mm}\partial\Omega}=\alpha n$.
In order to get an expression of the Riemannian shape gradient with respect to the Sobolev metric $g^1$, we look at the isomorphism (\ref{isomorphismTcBe}). Due to this isomorphism, a tangent vector $h\in T_\Gamma B_e$ is given by $h=\alpha n$ with $\alpha\in \mathcal{C}^\infty(\Gamma)$. This leads to the following definition.

\begin{definition}[Riemannian shape gradient with respect to the \begin{rrev}first\end{rrev} Sobolev metric]
	\label{definition_shapegradient}
	A Riemannian representation of the shape derivative, i.e., the Riemannian shape gradient of a shape differentiable objective function $J$ in terms of the first Sobolev metric~$g^1$, is given by
	\begin{equation}
	\label{RiemannianShapeGradient}
	\text{\emph{grad}}(J)=qn \ \text{ with } \ (I-A\kathrin{D^2_s})q\kathrin{n}=r\kathrin{n},
	\end{equation}
	where \kathrin{$D_s$ is the arc length derivative with respect to $\Gamma\in B_e$}, $A>0$, $q\in \mathcal{C}^\infty(\Gamma)$ and $r$ denotes \begin{rrev}the function in the shape derivative representation (\ref{bound_form}) for which we assume\end{rrev} \kathrin{$r\in \mathcal{C}^\infty(\Gamma)$}.
\end{definition}

Next, we specify the \emph{Riemannian shape Hessian} \begin{rrev}with respect to the first Sobolev metric\end{rrev}. It is based on the Riemannian connection $\nabla$ related to the Sobolev metric $g^1$ \begin{rrev}given in one of the main theorems of this paper, Theorem~\ref{theorem_covdervH1}, as well as on the Riemannian shape gradient definition, Definition~\ref{definition_shapegradient}\end{rrev}. 
In analogy to \cite{Absil}, we can define the Riemannian shape Hessian as follows:

\begin{definition}[Riemannian shape Hessian \begin{rrev}with respect to the first Sobolev metric\end{rrev}]
	\label{definition_Hessian}
	\kathrin{Let $\nabla$ be the covariant derivative associated with the  Sobolev metric $g^1$.}
	The Riemannian shape Hessian \kathrin{with respect to the Sobolev metric $g^1$} of a two times shape differentiable objective function $J$ is defined as the linear mapping
	\begin{equation}
	\label{Hessian_Sobolev}
	T_\Gamma B_e \to T_\Gamma B_e \text{, } h \mapsto \text{\emph{Hess}}(J)[h]:= \nabla_h \text{\emph{grad}}(J).
	\end{equation}
\end{definition} 

The Riemannian shape gradient and the Riemannian shape Hessian with respect to the Sobolev metric $g^1$ are required to \begin{rrev}apply first and second order optimization methods\end{rrev} in the shape space $(B_e,g^1)$. 
\begin{rrev}
	The gradient method is an example for a first order optimization method. If we apply the gradient method to (\ref{minproblem}) and consider $g^1$ on $B_e$, we need to compute the Riemannian shape gradient with respect to $g^1$ from (\ref{RiemannianShapeGradient}). The negative gradient is then used as descent direction for the objective functional $J$.
	An example for a second order method is the Newton method. 
	If we apply the Newton method to (\ref{minproblem}), we need to solve---similarly to standard non-linear programming---the\end{rrev}
\kathrin{
	problem of finding $\Gamma\in B_e$ with 
	\begin{equation}\label{Newton-problem}
	\text{grad}\hspace{.3mm}(J(\Omega))=0,
	\end{equation}
	where $\Gamma$ denotes the boundary of $\Omega$ and $\text{grad}\hspace{.3mm}J$ is the gradient with respect to $g^1$ (cf. (\ref{RiemannianShapeGradient})). 
	
	\begin{rrev}
		In general, the calculations of optimization methods on manifolds have to be performed in tangent spaces. 
		This means, points from a tangent space have to be mapped to the manifold in order to get a new iterate.
		\begin{rrrev}
			More precisely, we need to take a given tangent vector to the manifold, run along the geodesic starting at that point and go in that direction for a special length defined by the optimization process.
		\end{rrrev}
		The computation of the \rrevkathrin{Riemannian} exponential map, which is the theoretically superior choice of such a mapping, is prohibitively expensive in the most applications. 
		However, in \cite{Absil}, it is shown that a so-called \emph{retraction} is a first-order approximation and sufficient.
	\end{rrev}

	\begin{definition}[Retraction]
		\label{definition_retraction}
		A retraction on a manifold $M$ is a smooth mapping $\mathcal{R}\colon TM\to M$ with the following properties:
		\begin{itemize}
			\item[(i)] $\mathcal{R}_p(0_p)=p$, where $\mathcal{R}_p$ denotes the restriction of $\mathcal{R}$ to $T_pM$ and $0_p$ denotes the zero element of $T_pM$.
			\item[(ii)] $d\mathcal{R}_p(0_p)=\text{id}_{T_pM}$, where $\text{id}_{T_pM}$ denotes the identity mapping on $T_pM$ and $d\mathcal{R}_p(0_p)$ denotes the pushforward of $0_p\in T_pM$ by $\mathcal{R}$.
		\end{itemize}
	\end{definition}
	
	\noindent
	For example, in $B_e(S^1,\mathbb{R}^2)$, \begin{rrrev} for sufficiently small perturbations $\alpha \in \mathcal{C}^\infty(S^1)$, a retraction $\mathcal{R}$ is defined by
		\begin{equation}
		\label{retraction}
		\mathcal{R}_c(\eta_c):= c+\eta_c,
		\end{equation}
		where $\eta_c\in T_cB_e(S^1,\mathbb{R}^2)$ and  $
		c+\eta_c\colon S^1 \to \mathbb{R}^2,\,
		\theta \mapsto c(\theta)+\alpha(\theta)n(c(\theta))
		$.
	\end{rrrev}
	
	\begin{rrev}
		Now, we are able to formulate the gradient in $(B_e,g^1)$ (cf.~Algorithm~\ref{Algo:Gradient_g1}). 
		Thanks to the definition of the Riemannian shape Hessian $\text{Hess}\hspace{.3mm}J$ in $(B_e,g^1)$, which is based on the resulting Riemannian connection $\nabla$ given in (\ref{cov_der}), we can formulate also the Newton method in $(B_e,g^1)$ (cf.~Algorithm~\ref{Algo:Newton}).
		We require the shape function $r$ of the surface shape derivative (cf.~(\ref{bound_form})) in both algorithms. This function is needed to compute the shape gradient with respect to $g^1$ in each iteration. 
		In~Algorithm~\ref{Algo:Newton},\end{rrev}
	both, the Riemannian shape gradient and the Riemannian shape Hessian \begin{rrev}with respect to $g^1$\end{rrev}, are required. 
	If we have a PDE constrained shape optimization problem, the Newton method can be applied to find stationary points of the Lagrangian of the optimization problem which leads to the Lagrange-Newton method.
	
	\begin{algorithm}
		\begin{rrev}
			\caption{Gradient method in $(B_e,g^1)$}
			\label{Algo:Gradient_g1}
			\begin{algorithmic}
				\State \textbf{Require:} Objective function $J$ on $(B_e,g^1)$; retraction $\mathcal{R}$ on $(B_e,g^1)$; function $r$ of the surface shape derivative (cf.~(\ref{bound_form})).
				\vspace{.1cm}
				\State \textbf{Goal:} Find the solution of $\min\limits_{\xi\in B_e}J(\xi)$.
				\State \textbf{Input:} Initial shape $\xi^0\in B_e$. 
				\vspace{.3cm}
				
				\State \textbf{for} $k=0,1,\dots$ \textbf{do}
				\vspace{.1cm}
				\State Set 
				\begin{equation}
				\label{Update_g1}
				\xi^{k+1}:= \mathcal{R}_{\xi^k}\left(-\alpha^k \text{grad}(J(\xi^k))\right)
				\end{equation}
				for some steplength $\alpha^k$, where  $\text{grad}(J(\xi^k))$ denotes the shape gradient of $J$ in $\xi^k$ with respect to $g^1$.
				\vspace{.1cm}
				\State \textbf{end for}
				\vspace{.3cm}
			\end{algorithmic}
		\end{rrev}
	\end{algorithm}

	\begin{algorithm}
		\begin{rrev}
			\caption{Newton method in $(B_e,g^1)$}
			\label{Algo:Newton}
			\begin{algorithmic}
				\State \textbf{Require:} Objective function $J$ on $(B_e,g^1)$; retraction $\mathcal{R}$ on $(B_e,g^1)$; function $r$ of the surface shape derivative (cf.~(\ref{bound_form})); covariant derivative $\nabla$ on $(B_e,g^1)$.
				\vspace{.1cm}
				\State \textbf{Goal:} Find the solution of $\min\limits_{\xi\in B_e}J(\xi)$.
				\State \textbf{Input:} Initial shape $\xi^0\in B_e$.
				\vspace{.3cm}
				
				\State \textbf{for} $k=0,1,\dots$ \textbf{do}
				\vspace{.1cm}
				\State [1] Compute the increment $\Delta \xi^k $ as solution of the so-called \emph{Newton equation}
				\begin{equation}\label{newton}
				\text{Hess}(J(\xi^k))[\Delta \xi^k]=-\text{grad}(J(\xi^k)),
				\end{equation}
				where $\text{grad}(J(\xi^k))$ and $\text{Hess}(J(\xi^k))$ denote the shape gradient and the shape Hessian of $J$ in $\xi^k$ with respect to $g^1$, respectively.
				\vspace{.1cm}
				\State [2] Set $
				\xi^{k+1}:= \mathcal{R}_{\xi^k}(\Delta \xi^k)$.
				\vspace{.1cm}
				\State \textbf{end for}
				\vspace{.3cm}
			\end{algorithmic}
		\end{rrev}
	\end{algorithm}
	
}

\paragraph{Optimization based on Steklov-Poincar\'{e} metrics.}
\label{SubsectionSP}

\begin{rrev}
	Gradients with respect to $g^1$ are based on surface expressions of shape derivatives as you can see in (\ref{RiemannianShapeGradient}), where $r$ is the function in the surface shape derivative representation (\ref{bound_form}). 
	As outlined at the beginning of this section, volume expressions are preferable over surface forms.
\end{rrev} 
One possible approach to use volume forms is to consider \emph{Steklov-Poincar\'{e} metrics} $g^S$ (cf.~\cite{schulz2015Steklov}).
\begin{rrev}
	In the following, we summarize some of the main results related to this metric from \cite{schulz2015Steklov} with view on first order optimization approaches in the space of smooth shapes. In order to be able to formulate higher order methods in $(B_e,g^S)$ an explicit expression of the covariant derivative with respect to the Sobolev metric is necessary. The derivation of such an expression and also the formulation and investigation of higher order methods in $(B_e,g^S)$ is not in the scope of this paper and left for future work.
\end{rrev}

\kathrin{In the following, we need to deal with Lipschitz boundaries. Since there are several competing conditions which are
	used to define a Lipschitz boundary, we first specify its definition:
	\begin{definition}[$\mathcal{C}^{k,r}$-boundary, Lipschitz boundary]
		Let $\Omega\subset\mathbb{R}^d$ be open with boundary $\Gamma=\partial\Omega$. 
		Moreover, let $k\in \overline{\mathbb{N}}$ and $ \mathcal{C}^{k,r}(\overline{\Omega})$ denote the set of $\mathcal{C}^k$-functions which are Hölder-continuous with exponent $r\in[0,1]$.
		Further, $B_d(x,R)$ denotes the ball in $\mathbb{R}^d$ centered at $x\in\mathbb{R}^d$ with radius $R>0$.
		We say $\Omega$ has a $\mathcal{C}^{k,r}$-boundary or $\Omega$ is $\mathcal{C}^{k,r}$ if for any $x\in\Gamma$ there exist
		local coordinates $y_1,...,y_d$ centered at $x$, i.e., such that $x$ is the unique solution of $y_1=\dots=y_d=0$, and constants $a,b>0$ as well as a mapping $\psi\in \mathcal{C}^{k,r}( B_{d-1}(x,a))$, where $B_{d-1}(x,a)$ is considered in the linear subspace defined by $(y_1,...,y_{d-1})$, subject to the following conditions:
		\begin{itemize}
			\item[(i)] $y_d=\psi(\widetilde{y})  \Rightarrow  (\widetilde{y},y_d)\in\Gamma$,
			\item[(ii)] $\psi(\widetilde{y})<y_d<\psi(\widetilde{y})+b  \Rightarrow  (\widetilde{y},y_d)\in\Omega$,
			\item[(iii)] $\psi(\widetilde{y})-b<y_d<\psi(\widetilde{y})  \Rightarrow  (\widetilde{y},y_d)\not\in\overline{\Omega}$.
		\end{itemize}
	\end{definition}
}

\begin{definition}[Steklov-Poincar\'{e} metric]
	\label{definition_Steklovme}
	\kathrin{Let $\Omega\subset X\subset\mathbb{R}^d$ be a compact domain with $\Omega\neq\emptyset$ and \begin{rrev}Lipschitz\end{rrev}-boundary $\Gamma:=\partial \Omega$, where $X$ denotes a bounded domain with Lipschitz-boundary $\Gamma_\text{\emph{out}}:=\partial X$.}
	The Steklov-Poincar\'{e} metric is given by
	\begin{equation}\label{definition_Steklovmetric}
	\begin{split}
	g^S\colon H^{1/2}(\Gamma)\times H^{1/2}(\Gamma) & \to \mathbb{R},\\
	(\alpha,\beta) &\mapsto %\langle\alpha,(S^{pr})^{-1}\beta\rangle=
	\int_{\Gamma} \alpha(s)\cdot [(S^{pr})^{-1}\beta](s)\ ds.
	\end{split}
	\end{equation}
	Here $S^{pr}$ denotes the projected Poincar\'e-Steklov operator which is given by
	\begin{equation}
	S^{pr}\colon  H^{-1/2}(\Gamma_\text{\emph{out}}) \to H^{1/2}(\Gamma_\text{\emph{out}}),\
	\alpha \mapsto (\gamma_0 U)^T n,
	\end{equation}
	where $\gamma_0\colon  H^1_0(X,\mathbb{R}^d) \to H^{1/2}(\Gamma_\text{\emph{out}},\mathbb{R}^d)$, $U \mapsto U\,\rule[-2mm]{.1mm}{4mm}_{\, \Gamma_\text{\emph{out}}}$ and $U\in H^1_0(X,\mathbb{R}^d)$ solves the Neumann problem
	\begin{equation}\label{weak-elasticity-N2}
	a(U,V)=\int_{\Gamma_\text{\emph{out}}} \alpha\cdot (\gamma_0 V)^T n\ ds\quad \forall\hspace{.3mm}  V\in H^1_0(X,\mathbb{R}^d)
	\end{equation}
	with $a(\cdot,\cdot)$ being a symmetric and coercive bilinear form.
\end{definition}

\begin{remark}
	Note that a Steklov-Poincar\'{e} metric depends on the choice of the bilinear form. Thus, different bilinear forms lead to various Steklov-Poincar\'{e} metrics.
\end{remark}

Next, we state the connection of $B_e$ with respect to the Steklov-Poincar\'{e} metric $g^S$ to shape calculus.
As already mentioned, the shape derivative can be expressed as the surface integral (\ref{bound_form})
due to the Hadamard Structure Theorem. Recall that the shape derivative can be written more concisely (cf.~(\ref{concisely_expression})).
Due to isomorphism~(\ref{isomorphismTcBe}) and expression~(\ref{concisely_expression}), we can state the connection of the shape space $B_e$ with respect to the Steklov-Poincar\'{e} metric $g^S$ to shape calculus. % and get the following definition. 

\begin{definition}[Shape gradient with respect to Steklov-Poincar\'{e} metric]
	\label{Def:Steklov}
	Let $r\in \mathcal{C}^\infty(\Gamma)$ denote the \begin{rrev}function in the shape derivative expression\end{rrev} (\ref{bound_form}). Moreover, let $S^{pr}$ be the projected Poincar\'e-Steklov operator and let $\gamma_0$ be as in Definition~\ref{definition_Steklovme}. A representation $h\in T_{\Gamma} B_e\cong\mathcal{C}^\infty(\Gamma)$ of the shape gradient in terms of $g^S$ is determined by
	\begin{equation}
	g^S(\phi,h)=\left(r,\phi\right)_{L^2(\Gamma)} \quad \forall \phi\in \mathcal{C}^\infty(\Gamma),
	\end{equation}
	which is equivalent to
	\begin{equation}
	\label{equation_sgSteklov}
	\int_{\Gamma} \phi(s)\cdot [(S^{pr})^{-1}h](s) \ ds=\int_{\Gamma} r(s)\phi(s) \ ds \quad  \forall \phi\in \mathcal{C}^\infty(\Gamma).
	\end{equation}
\end{definition}

\kathrin{
	\begin{remark}
		In Definition~\ref{Def:Steklov}, the isomorphism $T_{\Gamma} B_e\cong\mathcal{C}^\infty(\Gamma)$  is given. It is worth to mention that for example identifying $\Gamma$ with the corresponding embedding of the circle leads to this isomorphism. In particular, attention needs to be put onto \begin{rrev}(\ref{isomorphismTcBe}) and\end{rrev} (\ref{TangentSpace}).
	\end{remark}
}

Now, the shape gradient with respect to Steklov-Poincar\'{e} metric is defined. This enables the formulation of optimization methods in $B_e$ which involve volume formulations of shape derivatives.
From (\ref{equation_sgSteklov}) we get $h=S^{pr}r=(\gamma_0 U)^T n$, where $U\in H^1_0(X,\mathbb{R}^d)$ solves
\begin{equation}
\label{eq_most_important}
a(U,V)=\int_{\Gamma} r\cdot (\gamma_0 V)^T n \ ds
=DJ_{\Gamma}[V]=DJ_\Omega[V] \quad \forall \hspace{.3mm} V\in H^1_0(X,\mathbb{R}^d)
\end{equation}
\begin{rrev}
	with $a(\cdot,\cdot)$ being the symmetric and coercive bilinear form on $H_0^1(X,\mathbb{R}^d) \times H_0^1(X,\mathbb{R}^d)$ of the Steklov-Poincar\'{e} metric definition (cf.~(\ref{weak-elasticity-N2})).
\end{rrev}
\begin{rrev}
	The identity (\ref{eq_most_important}) opens the door to consider volume expression of shape derivatives to compute the shape gradient with respect to $g^S$.
	In order to compute the shape gradient, 
\end{rrev} 
we have to solve 
\begin{equation}
a(U, V) = b(V) \begin{rrev}\quad \forall \hspace{.3mm} V\in H^1_0(X,\mathbb{R}^d)\end{rrev}
\label{deformatio_equation}
\end{equation}
with $b(\cdot)$ being a linear form and given by
$$b(V):=DJ_\text{vol}(\Omega)[V]+DJ_\text{surf}(\Omega)[V].$$
Here $J_\text{surf}(\Omega)$ denotes parts of the objective function leading to surface shape derivative expressions, e.g., perimeter regularizations, and is incorporated as Neumann boundary condition \begin{rrev}in equation (\ref{deformatio_equation})\end{rrev}. Parts of the objective function leading to volume shape derivative expressions are denoted by $J_\text{vol}(\Omega)$. 
The \kathrin{bilinear form} $a(\cdot,\cdot)$ can be chosen, e.g., as the weak form of the linear elasticity equation. \begin{rrev}More details can be found below, in the paragraph about the comparison of Algorithm \ref{Algo:Gradient_g1} and Algorithm \ref{Algo:Gradient}.\end{rrev} 

\begin{remark}
	Note that it is not ensured that $U\in H^1_0(X,\mathbb{R}^d)$ is $\mathcal{C}^\infty$. Thus, $h=S^{pr}r=(\gamma_0 U)^\top n$ is not necessarily an element of $T_{\Gamma}B_e$.
	However, under special assumptions depending on the coefficients of a second-order partial differential operator and the right-hand side of a PDE, a weak solution $U$ which is at least $H^1_0$-regular is $\mathcal{C}^\infty$ (cf.~\cite[Section~6.3, Theorem~6]{Evan}).
\end{remark}

\begin{rrev}
	Thanks to the definition of the gradient with respect to $g^S$ we are able to formulate the gradient method on $(B_e,g^S)$ (cf.~Algorithm~\ref{Algo:Gradient}).
	We compute the Riemannian shape gradient with respect to $g^S$ from (\ref{deformatio_equation}). The negative solution $-U$ is then used as descent direction for the objective functional $J$. 
	In Algorithm~\ref{Algo:Gradient}, in order to be in line with the above theory, it is assumed that in each iteration $k$, the shape $\xi^k$ is a subset of a general surrounding space $X$, which is assumed to be a bounded domain with Lipschitz boundary as illustrated in Figure~\ref{fig:Domain}. 
\end{rrev}

\begin{algorithm}
	\begin{rrev}
		\caption{Gradient method in $(B_e,g^S)$}
		\label{Algo:Gradient}
		\begin{algorithmic}
			\State \textbf{Require:} Objective function $J$ on $(B_e,g^S)$; retraction $\mathcal{R}$ on $(B_e,g^S)$; symmetric and coercive bilinear form $a(\cdot,\cdot)$; shape derivative $dJ(\cdot)[\cdot]$ in volume or surface form or a combination of both; surrounding bounded domain $X$ with Lipschitz boundary, which includes all shape iterates.
			\vspace{.1cm}
			\State \textbf{Goal:} Find the solution of $\min\limits_{\xi\in B_e}J(\xi)$.
			\State \textbf{Input:} Initial shape $\xi^0\in B_e$ with $\xi^0 \subset X$. 
			\vspace{.3cm}
			
			\State \textbf{for} $k=0,1,\dots$ \textbf{do}
			\vspace{.1cm}
			\State [1] Solve $a(U^k,V)=dJ(\xi^k)[V] $ for all vector fields $V$ with $H^1_0$-regularity on $X$.
			\vspace{.1cm}
			\State [2] Set 
			\begin{equation}
			\label{Update_gS}
			\xi^{k+1}:= \mathcal{R}_{\xi^k}(-\alpha^k ((\gamma_0 U^k)^T n^k))
			\end{equation}
			for some steplength $\alpha^k$, where $n^k$ denotes the unit outward normal vector  to $\xi^k$ and $\gamma_0$ is the trace operator.
			\vspace{.1cm}
			\State \textbf{end for}
			\vspace{.3cm}
		\end{algorithmic}
	\end{rrev}
\end{algorithm}

\begin{rev}
	\paragraph{Comparison of Algorithm \ref{Algo:Gradient_g1} and Algorithm \ref{Algo:Gradient}.}
	\label{SubNum}
	
	We conclude this section with a brief discussion about \begin{rrev}
		Algorithm \ref{Algo:Gradient_g1} and Algorithm \ref{Algo:Gradient}.
		Since this paper does not focus on numerical investigations, we pick an example which is already implemented in \cite{schulz2015Steklov,Welker}. We summarize briefly numerical results observed in \cite{schulz2015Steklov,Welker} in order to
	\end{rrev}
	illustrate how the algorithms work. \begin{rrev}Additionally, we discuss the main differences between the two approaches.\end{rrev}
	
	\begin{figure}
		\vspace*{.3cm}
		\begin{center}
			\begin{overpic}[width=.45\textwidth]{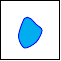}
				\put(40,40){$\Omega$}
				\put(10,13){$X\setminus \overline{\Omega}$}
				\put(38,103){$\Gt $}
				\put(43,74){$\color{blue}\Gi $}
				\put(-11,47){$\Gl $}
				\put(101,47){$\Gr $}
				\put(35,-8){$\Gb $}
				%  \put(65.5,54){$n$}
			\end{overpic}
		\end{center}
		\vspace*{.3cm}
		\caption{Example of the domain $X$ with $\Omega\subset X\subset\mathbb{R}^2$.}
		\label{fig:Domain}
	\end{figure}
	
	Let  $\Omega\subset X\subset \mathbb{R}^2$ be a domain with $ \partial\Omega = \Gi$, where $X$ denotes a bounded domain with Lipschitz-boundary $\Gamma_\text{out}:=\partial X$. In contrast to the outer boundary $\Go$, which is assumed to be fixed and partitioned in $\Go:=\Gb\sqcup\Gl\sqcup\Gr\sqcup\Gt$ (here, $\sqcup$ denotes the disjoint union), the inner boundary $\Gi$, which is also called the interface, is variable. Let the interface $\Gi$ be an element of $B_e(S^1,\mathbb{R}^2)$. 
	Note that $X$ depends on $\Gi$. Thus, we denote it by $X(\Gi)$. 
	Figure~\ref{fig:Domain} illustrates this situation. 
	We consider the following parabolic PDE constrained interface problem (cf.~\cite{schulz2015Steklov,Welker}):
	\begin{align}\label{oc1p}
	\hspace{-4.5cm}\min_{\Gi} \hspace{0,1cm} \intdxdt{(y-\overline{y})^2}+\mu\int_{\Gi}1\hspace{.5mm}ds
	\end{align}
	\vspace*{-.5cm}
	\begin{align}
	\label{oc2p}
	\mbox{s.t. } \td{y}- \mathrm{div}(k\nabla y)&=f\quad \text{in }X(\Gi)\times(0,T]
	\\
	\label{oc3p}
	\hspace{20mm}y&=1\quad \text{on }\Gt\times(0,T]
	\\
	\label{oc4p}
	\nd{y}&=0\quad \text{on }(\Gb\cup\Gl\cup\Gr)\times(0,T]
	\\
	\label{oc5p}
	y&=y_0\quad\text{in }X(\Gi)\times\{0\}
	\end{align}
	with 
	\begin{equation*}
	k:=\begin{cases}
	k_1 = \mathrm{const.}\quad\text{ in }X\setminus \overline{\Omega}\times(0,T]\\
	k_2 = \mathrm{const.} \quad\text{ in }\Omega\times(0,T]
	\end{cases}
	\end{equation*}
	denoting a jumping coefficient, $n$ being the unit outward normal vector to $\Omega$ and $\bar{y}\in H^1(X(\Gi))$ represents data measurements. 
	The second term in the objective function (\ref{oc1p}) is a perimeter regularization with $\mu>0$.
	Please note that formulation (\ref{oc2p}) of the PDE has to be understood only formally because of the jumping coefficient $k$.
	
	\begin{rrev}
		In order to solve the shape optimization problem (\ref{oc1p})-(\ref{oc5p}), we first need to solve the underlying PDE, the so-called \emph{state equation}.
	\end{rrev}
	The solution of the parabolic boundary value problem (\ref{oc2p})-(\ref{oc5p}) is obtained by discretizing its weak formulation with standard linear finite elements in space and an implicit Euler scheme in time.
	The diffusion parameter $k$ is discretized as a piecewise constant function. 
	Figure~\ref{fig_bfgs_deformations} illustrates an example initial shape geometry, where the domain is discretized with a fine and coarse finite element mesh. 
	\begin{rrev}Besides the underlying PDE, we also need to solve the corresponding adjoint problem to the shape optimization problem (\ref{oc1p})-(\ref{oc5p}), which is given in our example\end{rrev} by
	\begin{align}
	-\frac{\partial p}{\partial t}-\mathrm{div}(k\nabla p)\hspace{.3mm}&=-(y-\overline{y}) \quad \text{in }X(\Gi) \times [0,T)\label{adjoint1p}\\
	p\hspace{.3mm}&= 0\quad\text{in }X(\Gi)  \times \{T\}\label{adjoint2p}\\
	\nd{p}\hspace{.3mm}&=0\quad\text{on }\left(\Gb\cup\Gl\cup\Gr\right) \times [0,T)\label{adjoint5p}\\
	p\hspace{.3mm}&=0\quad\text{on }\Gt \times [0,T)\label{adjoint6p}%\\
	\end{align}
	and which can be discretized in the same way as the state equation.
	
	\begin{rrev}
		\begin{remark}
			In general, the solution of the state and adjoint equation are needed in Algorithm~\ref{Algo:Gradient_g1} and Algorithm~\ref{Algo:Gradient} because they are part of the shape derivative of the objective functional. 
		\end{remark}
	\end{rrev}
	
	\begin{rrev}
		
		We use the retraction given in (\ref{retraction}) in order to update the shapes according to Algorithm~\ref{Algo:Gradient_g1} (cf.~(\ref{Update_g1})) and Algorithm~\ref{Algo:Gradient} (cf.~(\ref{Update_gS})), respectively. 
		This retraction is closely related to the perturbation of identity defined on the domain $X$. Given a stating shape $\Gamma^k$ in the $k$-th iteration of Algorithm~\ref{Algo:Gradient}, the perturbation of identity acting on the domain $X$ in the direction $U^k$, where $U^k$ solves (\ref{deformatio_equation}), gives
		\begin{equation}
		X(\Gamma^{k+1})=\{x\in X\colon x=x^k+t^kU^k\},
		\end{equation}
		i.e. the vector field $U^k$ weighted by a step size $t^k$ is added as a deformation to all nodes in the finite element mesh. One calls $U^k$ also \emph{mesh deformation (field)}. Here, the volume form allows us to optimize directly over the domain $X$ containing $\Gamma^k\in B_e$. It is worth to mention that, in practice, we are only interested in the deformation on $X$ because we need to update the finite element mesh after each iteration. The update of the shape $\Gamma^k$ itself is contained in this deformation field.
		In contrast to Algorithm~\ref{Algo:Gradient}, Algorithm~\ref{Algo:Gradient_g1} can only work with surface shape derivative expressions. 
		These surface formulations would give us descent directions (in normal directions) for $\Gamma^k$ only, which would not help us to move mesh elements around the shape. Additionally, when we are working with a surface shape derivative, we need to solve another PDE in order to get a mesh deformation in the ambient space $X$. Below, this issue is addressed in more detail.

		Both approaches follow roughly the same steps but with
		a major difference in the way of computing the mesh deformation.  
	\end{rrev}
	For convenience we summarize \begin{rrev}one optimization iteration and\end{rrev} the main aspects of the two approaches: 
	
	\begin{rrev}
		\begin{itemize}
			\itemsep5pt
			\item[1.] Solve the state and adjoint equation.
			\item[2.] Compute the mesh deformation:
			\vspace*{.1cm}
			\begin{itemize}
				\item \emph{Algorithm \ref{Algo:Gradient}}: 
				The computation of a representation of the shape gradient with respect to the chosen inner product of the tangent space is moved into the mesh deformation itself. In particular,  we get the gradient representation and the mesh deformation all at once from (\ref{eq_most_important}), which is very attractive from a computational point of view.
				The bilinear form $a(\cdot,\cdot)$ in (\ref{deformatio_equation}) is used as both, an inner product and a mesh deformation, leading to only one linear system, which has to be solved. In practice, the bilinear form $a(\cdot,\cdot)$ in (\ref{deformatio_equation})  is chosen as the weak form corresponding to the linear elasticity equation, i.e., 
				$$ a(U,V)=  \int_{X(\Gamma)} \sigma(U):\epsilon(V) \, dx, $$
				where $:$ denotes the sum of the component-wise products and $\sigma$, $\varepsilon$  are the so-called \emph{strain} and \emph{stress tensor}\footnote{Please see Remark~\ref{Lame} for the definition of the strain and stress tensor.}, respectively.
				In strong form, (\ref{deformatio_equation}) is given by
				\begin{align}
				\label{eq_linelas1}
				\text{div}( \sigma ) &= f^\text{elas} \quad \text{in} \quad X(\Gi)\\
				U &= 0 \quad \text{on} \quad \Go\label{eq_linelas2}
				\end{align}
				Here, the source term $f^\text{elas} $ in its weak form is given by the shape derivative parts in volume form, where parts of the objective function leading to surface expressions only, such as, for instance, the perimeter regularization, are incorporated in Neumann boundary conditions. In our example, the shape derivative is given by
				\begin{equation*}
				\label{boundary_expressionp}
				\begin{split}
				DJ(\Gamma)[V]
				=&\int_{0}^{T}\int_{X(\Gamma)}-k\nabla y^T\left(\nabla V+\nabla V^T\right)\nabla p-p\nabla f^T V\\
				&\phantom{\int_{0}^{T}\int_{X(\Gamma)}}+\mathrm{div}(V)\left(\frac{1}{2}(y-\overline{y})^2+\td{y}p+k\nabla y^T\nabla p-fp\right)dx\hspace{.3mm}dt\\
				&+\int_{\Gamma}\kappa \left<V,n\right>ds
				\end{split}
				\end{equation*}
				where $\kappa$ denotes the mean curvature of $\Gamma$ and $y, \,p$ denote the solution of the state and adjoint equation, respectively
				(cf.~\cite{schulz2014structure}).
				\item \emph{Algorithm \ref{Algo:Gradient_g1}:} 
				First, a representation of the shape gradient on $\Gi$ with respect to the Sobolev metric $g^1$ as given in (\ref{RiemannianShapeGradient}) needs to be computed by solving
				\begin{equation}
				(I-AD^2_s)qn=rn.
				\end{equation}
				In our example, $r$ is given by $r=\left\llbracket k \right\rrbracket\nabla y_1^T\nabla p_2n+\kappa n$, where $\kappa$ denotes the mean curvature of $\Gamma$, the jump symbol $\left\llbracket\cdot\right\rrbracket$ is defined on the interface $\Gamma$ by $\llbracket k \rrbracket := k_1-k_2$, $y_{1} := \text{tr}_{\text{out}}(y\vert_{X\setminus \overline{\Omega}})$ with $y$ denoting the solution of the state equation, $p_2 := \text{tr}_{\text{in}}(p \vert_{\Omega})$ with $p$ denoting the solution of the adjoint equation, and $\text{tr}_{\text{in}}\colon \Omega \rightarrow \Gamma$ and $\text{tr}_{\text{out}}\colon X\setminus \overline{\Omega} \rightarrow \Gamma$ are trace operators (cf.~\cite{schulz2014structure}).
				In order to compute a mesh deformation field, we need to solve a further PDE. In practice, this further PDE is again equation (\ref{eq_linelas1})-(\ref{eq_linelas2}) but modified as follows: the Dirichlet boundary condition
				\begin{equation*}
				U = U^\text{surf} \quad \text{on} \quad \Gi
				\end{equation*}
				is added to (\ref{eq_linelas1})-(\ref{eq_linelas2}),
				where $U^\text{surf}$ is the representation of the shape gradient with respect to the Sobolev metric $g^1$; the source term $f^\text{elas}$ is set to zero.
			\end{itemize}	
			\item[3.] Apply the resulting deformation to the current finite element mesh, and go to the next iteration.
		\end{itemize}
	\end{rrev}

	\begin{rrev}
		\begin{remark}
			\label{Lame}
			The strain and stress tensor in (\ref{eq_linelas1}) are defined by	$\sigma  :=\lambda \text{tr}(\varepsilon) I + 2 \mu \varepsilon$, $\varepsilon := \frac{1}{2}\left(\nabla U + \nabla U^T\right)$, where $\lambda$ and $\mu$ denote the so-called \emph{Lam\'{e} parameters}.
			The Lam\'{e} parameters  do not need to have a physical meaning here but it is rather essential to understand their effect on the mesh deformation. 
			They can be expressed
			in terms of Young's modulus $E$ and Poisson's ratio $\nu$ as $\lambda = \frac{\nu E}{(1+\nu)(1-2\nu)} ,\,
			\mu  = \frac{E}{2(1+\nu)}$.
			Young's modulus $E$ states the stiffness of the material, which enables to control the step size for the shape update, and Poisson's ratio $\nu$ gives the ratio controlling how much the mesh expands in the remaining coordinate directions when compressed in one particular direction.
		\end{remark}
	\end{rrev}

	Besides saving analytical effort during the calculation process of the shape derivative, Algorithm~\ref{Algo:Gradient} is computationally more efficient than using  Algorithm~\ref{Algo:Gradient_g1}. 
	The optimization algorithm based on domain shape derivative expressions (Algorithm~\ref{Algo:Gradient}) can be applied to very coarse meshes with approximately $100\,000$ cells (cf. right picture of Figure \ref{fig_bfgs_deformations}).
	This is due to the fact that there is no dependence on normal vectors like in the case of surface shape gradients, \begin{rrev}
		which are needed in Algorithm~\ref{Algo:Gradient_g1}.
	\end{rrev}
	In \cite{schulz2015Steklov,Welker}, the convergence of the gradient method for the surface and volume shape derivative formulation are investigated for the parabolic shape interface problem.
	It can be observed that the convergence with the representation of the shape gradient with respect to $g^1$ seems to require fewer iterations compared to the domain-based formulation.
	Yet, the domain-based form is computationally more attractive since it also works for much coarser discretizations.
	This can be seen in a comparision the two meshes in Figure \ref{fig_bfgs_deformations}. In particular, the mesh in the left picture in Figure \ref{fig_bfgs_deformations} shows the necessary fineness of the mesh for the surface gradient (Algorithm~\ref{Algo:Gradient_g1}) to lead to a reasonable convergence.
	However, the coarse grid in Figure \ref{fig_bfgs_deformations} works only for the domain-based formulation (Algorithm~\ref{Algo:Gradient}).
	
	\begin{figure}
		\begin{center}
			\begin{tabular}{cc}
				\includegraphics[width=0.4\textwidth]{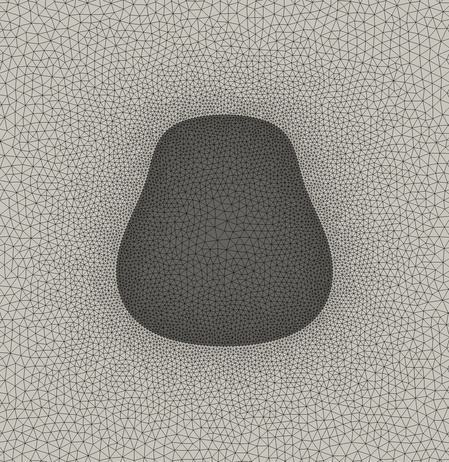}&
				\includegraphics[width=0.41\textwidth]{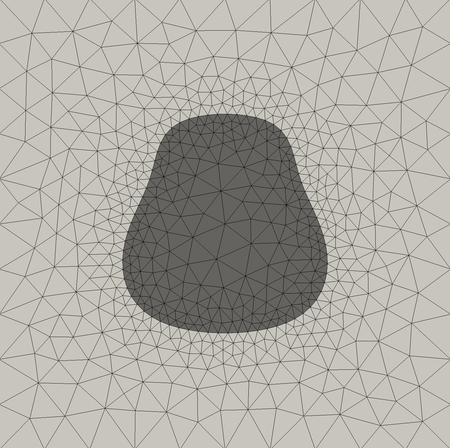}
			\end{tabular}
		\end{center}
		\caption{Different initial meshes}
		\label{fig_bfgs_deformations}
	\end{figure}

	\kathrin{We can conclude that} Algorithm~\ref{Algo:Gradient}  is very attractive from a computational point of view.
	However, the shape space $B_e$ containing smooth shapes unnecessarily limits the application of this algorithm.  
	More precisely, numerical investigations have shown that the optimization techniques also work on shapes with kinks in the boundary (cf.~\cite{SchulzSiebenborn,schulz2015Steklov,SiebenbornWelker_skin}). 
	This means that Algorithm~\ref{Algo:Gradient} is not limited to elements of $B_e$ and another shape space definition is required. Thus, in \cite{schulz2015Steklov}, the definition of smooth shapes is extended to so-called \emph. In the next section, it is clarified what we mean by $H^{1/2}$-shapes.  However, only a first try of a definition is given in \cite{schulz2015Steklov}. From a theoretical point of view there are several open questions about this shape space. The most important question is how the structure of this shape space is. If we do not know the structure, there is no chance to get control over the space. Moreover, the definition of this shape space has to be adapted and refined. The next section is concerned with the \begin{rrev}novel\end{rrev} space of $H^{1/2}$-shapes and in particular with its structure.
\end{rev}

%%%%%%%%%%%%%%%%%%%%%%%%%%%%%%%%%%%%%%%%%%%%%%%%%%%%%%%%%%%%%%%%%%%%%%%%%%%%%%%%%%%%%%%%%%%%%%%%%

\section{The shape space $\mathcal{B}^{\mathbf{1/2}}$}
\label{section_shapespace}

The Steklov-Poincar\'{e} metric correlates shape gradients with $H^1$-deformations. Under special assumptions, these deformations give shapes of class $H^{1/2}$, which are defined below. As already mentioned above the shape space $B_e$ unnecessarily limits the application of the methods mentioned in the previous section.
\kathrin{Thus, this section aims at a generalization of smooth shapes to shapes which arise naturally in shape optimization problems.}
In the setting of $B_e$, shapes can be considered as the images of embeddings.
From now on we have to think of shapes as boundary contours of deforming objects. Therefore, we need another shape space. 
In this section, we define the space of $H^{1/2}$-shapes and clarify its structure as a diffeological one. 

First, we do not only define diffeologies and related objects, but also explain the difference between \kathrin{diffeological spaces} and manifolds (Subsection~\ref{subsection_definitions_diffspace_B12}). 
\kathrin{In particular, we formulate the second main theorem of this paper, Theorem~\ref{theorem_diffman}.} 
Afterwards, the space of $H^{1/2}$-shapes is defined (Subsection~\ref{subsection_shapespaceB12}). \kathrin{In the third main theorem, Theorem~\ref{Theorem:DiffStructure},}  we see that it is a diffeological space.

\subsection{A brief introduction into diffeological spaces}
\label{subsection_definitions_diffspace_B12}

In this subsection, we define diffeologies and related objects. Moreover, we clarify the difference between manifolds and diffeological spaces. For a detailed introduction into diffeological spaces we refer to \cite{Iglesias}.

\subsubsection{Definitions}
\label{subsection_Def}

We start with the definition of a diffeological space and related objects like a \emph{diffeology}, with which a diffeological space is equipped, and \emph{plots}, which are the elements of a diffeology. Afterwards, we consider \emph{subset} and \emph{quotient diffeologies}. These two objects are required in the main theorem of Subsection \ref{subsection_shapespaceB12}. \kathrin{The definitions and theorems in this Subsection~\ref{subsection_Def} are summarized from \cite{Iglesias}.}

\begin{definition}[Parametrization, diffeology, diffeological space, plots]
	\label{Def:DiffSpace}
	Let $Y$ be a non-empty set. A parametrization in $Y$ is a map $U\to Y$, where $U$ is an open subset of $\mathbb{R}^n$.
	A diffeology on $Y$ is any set $D_Y$ of parametrizations in $Y$ such that the following three axioms are satisfied:
	\begin{itemize}
		\item[(i)] \textbf{Covering:} Any constant parametrization $\mathbb{R}^n\to Y$ is in $D_Y$.
		\item[(ii)] \textbf{Locality:} \kathrin{Let $P$ be a parametrization in $Y$, where $\text{dom}(P)$ denotes the domain of $P$. If, for all $r\in\text{dom}(P)$, there is an open neighborhood $V$ of $r$ such that the restriction $P\,\rule[-2mm]{.1mm}{4mm}_{\, V}\in D_Y$ , then  $P \in D_Y$.}
		\item[(iii)] \textbf{Smooth compatibility:} Let $p\colon O\to Y$ be an element of $D_Y$, where $O$ denotes an open subset of $\mathbb{R}^n$. Moreover, let $q\colon O'\to O$ be a smooth map in the usual sense, where $O'$ denotes an open subset of $\mathbb{R}^m$. Then $p\circ q\in D_Y$ holds.
	\end{itemize}
	A non-empty set $Y$ together with a diffeology $D_Y$ on $Y$ is called a diffeological space and denoted by $(Y,D_Y)$. The parametrizations $p\in D_Y$ are called plots of the diffeology $D_Y$. If a plot $p\in D_Y$ is defined on $O\subset \mathbb{R}^n$, then $n$ is called the dimension of the plot and $p$ is called $n$-plot.
\end{definition}

\noindent
In the literature, there are a lot of examples of diffeologies, e.g., the diffeology of the circle, the square, the set of smooth maps, etc. For those we refer to \cite{Iglesias}.

\begin{remark}
	A diffeology as a structure and a diffeological space as a set equipped with a diffeology are distinguished only formally. Every diffeology on a set contains the underlying set as the set of non-empty 0-plots (cf.~\cite{Iglesias}).
\end{remark}

Next, we want to connect diffeological spaces. This is possible though \emph{smooth maps} between two diffeological spaces. 

\begin{definition}[Smooth map between diffeological spaces, diffeomorphism]
	Let $(X,D_X),(Y,D_Y)$ be two diffeological spaces. A map $f\colon X\to Y$ is smooth if for each plot $p\in D_X$, $f\circ p$ is a plot of $ D_Y $, i.e., $f\circ D_X\subset D_Y$. If $f$ is bijective and if both, $f$ and its inverse $f^{-1} $, are smooth, $f$ is called a diffeomorphism. In this case, \kathrin{$(X,D_X)$} is called diffeomorphic to \kathrin{$(Y,D_Y)$}.
\end{definition}

The stability of diffeologies under almost all set constructions is one of the most striking properties of the class of diffeological spaces \kathrin{like in} the subset, quotient, functional or powerset diffeology. In the following, we concentrate on the subset and quotient diffeology. The concept of these are required in the proof of the main theorem in the next subsection. 

\paragraph{Subset diffeology.}
Every subset of a diffeological space carries a natural \emph{subset diffeology}, which is defined by the \emph{pullback} of the ambient diffeology by the \emph{natural inclusion}. 

Before we can construct the subset diffeology, we have to clarify the natural inclusion and the pullback. For two sets $A,B$ with $A\subset B$, the \emph{(natural) inclusion} is given by
$\iota_A\colon A\to B$, $x\mapsto x$. The pullback is defined as follows:

\begin{defthe}[Pullback]
	Let $X$ be a set and $(Y,D_Y)$ be a diffeological space. Moreover, $f\colon X\to Y$ denotes some map.
	\begin{itemize}
		\item[(i)] There exists a coarsest diffeology of $X$ such that $f$ is smooth. This diffeology is called the pullback of the diffeology $D_Y$ by $f$ and is denoted by $f^\ast(D_Y)$.
		\item[(ii)] Let $p$ be a parametrization in $X$. Then $p\in f^\ast(D_Y)$ if and only if $f\circ p\in D_Y$. 
	\end{itemize}
\end{defthe}

\emph{Proof.}
See \cite[Chapter~1, 1.26]{Iglesias}.\qed

\smallskip

The construction of subset diffeologies is related to so-called \emph{inductions}. 

\begin{definition}[Induction]
	Let $(X,D_X),(Y,D_Y)$ be diffeological spaces. A map $f\colon X\to Y$ is called induction if $f$ is injective and $f^\ast(D_Y)=D_X$, where $f^\ast(D_Y)$ denotes the pullback of the diffeology $D_Y$ by $f$.
\end{definition}

\noindent
The illustration of an induction as well as the criterions for being an induction can be found in \cite[Chapter~1, 1.31]{Iglesias}. 

Now, we are able to define the subset diffeology (cf.~\cite{Iglesias}).

\begin{defthe}[Subset diffeology]
	\label{def_subsetdiff}
	Let $(X,D_X)$ be a diffeological space and let $A\subset X$ be a subset. 
	Then $A$ carries a unique diffeology $D_A$, called the subset or induced diffeology, such that the inclusion map $\iota_A\colon A\to X$ becomes an induction, namely, $D_A=\iota_A^\ast(D_X)$.
	We call $(A,D_A)$ 
	the diffeological subspace of \kathrin{$(X,D_X)$}.
\end{defthe}

\paragraph{Quotient diffeology.}

Like every subset of a diffeological space inherits the subset diffeology, every quotient of a diffeological space carries a natural \emph{quotient diffeology} defined by the \emph{pushforward} of the diffeology of the source space to the quotient by the \emph{canonical projection}.

First, we have to clarify the canonical projection. 
For a set $A$ and an equivalence relation $\sim$ on $A$, the \emph{canonical projection} is defined as $\pi\colon X\to X/\hspace{-1mm}\sim$, $x\mapsto [x]$, where $[x]:= \{x'\in X\colon x\sim x'\}$ denotes the equivalence class of $x$ with respect to $\sim$.
Moreover, the pushforward has to be defined:

\begin{defthe}[Pushforward]
	\label{pushforward}
	Let $(X,D_X)$ be a diffeological space and $Y$ be a set. Moreover, $f\colon X\to Y$ denotes a map.
	\begin{itemize}
		\item[(i)] There exists a finest diffeology of $Y$ such that $f$ is smooth. This diffeology is called the pushforward of the diffeology $D_X$ by $f$ and is denoted by $f_\ast(D_X)$.
		\item[(ii)] A parametrization $p\colon U\to Y$ lies in $f_\ast(D_X)$ if and only if every point $x\in U$ has an open neighbourhood $V\subset U$ such that $p\,\rule[-2mm]{.1mm}{4mm}_{\, V}\colon V\to Y$ is constant or of the form $p\,\rule[-2mm]{.1mm}{4mm}_{\, V}=f\circ q$ for some plot $q\colon V\to X$ with $q\in D_X$.
	\end{itemize}
\end{defthe}

\emph{Proof.}
See \cite[Chapter~1, 1.43]{Iglesias}.\qed

\begin{remark}
	If a map $f$ from a diffeological space $(X,D_X)$ into a set $Y$ is surjective, then $f_\ast(D_X)$ consists precisely of the plots $p\colon U\to Y$ which locally are of the form $f\circ q$ for plots $q\in D_X$ since those already contain the constant parametrizations.
\end{remark}

The construction of quotient diffeologies is related to so-called \emph{subductions}.

\begin{definition}[Subduction]
	Let $(X,D_X),(Y,D_Y)$ be diffeological spaces. A map $f\colon X\to Y$ is called subduction if $f$ is surjective and $f_\ast(D_X)=D_Y$, where $f_\ast(D_X)$ denotes the pushforward of the diffeology $D_X$ by $f$.
\end{definition}

\noindent 
The illustration of a subduction as well as the criterions for being a subduction can be found in \cite[Chapter~1, 1.48]{Iglesias}. 

Now, we can define the quotient diffeology (cf.~\cite{Iglesias}).

\begin{defthe}[Quotient diffeology]
	\label{def_quotdiff}
	Let $(X,D_X)$ be a diffeological space and $\sim$ be an equivalence relation on $X$. 
	Then the quotient set $X/\hspace{-1mm}\sim$ carries a unique diffeologcial sturcture $D_{X/\sim} $, called the quotient diffeology, such that the canonical projection $\pi\colon X\to X/\hspace{-1mm}\sim$ becomes a subduction, namely, $D_{X/\sim}=\pi_\ast(D_X)$. 
	We call $\left(X/\hspace{-1mm}\sim,D_{X/\sim}\right)$ the diffeological quotient of \kathrin{$(X,D_X)$} by the relation $\sim$.
\end{defthe}

\kathrin{
	One aim of this paper is to go towards optimization algorithms in diffeological spaces.
	Thus, we end this subsection with a brief discussion about the topology of a diffeological space which is necessary to discuss properties of optimization methods in diffeological spaces like convergence. \begin{rrev}
		Every diffeological space induces a unique topology, the so-called \emph{$D$-topology}, which is a natural topology and introduced by Patrick Iglesias-Zemmour for each diffeological space (cf.~\cite{Iglesias}). 
	\end{rrev} 
	In particular, openess, compactness and convergence depend on the $D$-topology.
	Given a diffeological space $(X,D_X)$, 
	the $D$-topology is the finest topology such that all plots are continuous. That is, a subset $U$ of $X$ is open (in the D-topology) if for any plot $p\colon O\to X$ the pre-image $p^{-1}U\subset O$ is open.
	For more information about the $D$-topology we refer to the literature, e.g., \cite[Chapter~2, 2.8]{Iglesias} or \cite{Christensen}.
	However, if $(X,D_X)$ is a diffeological space and one knows that a sequence $\{x_n\}$ converges with respect to the topology of $X$, it is not guaranteed that $\{x_n\}$ converges also for the $D$-topology because this topology is finer than the given one on $X$. Thus, all discussions about compactness, convergence, etc in the diffeological sense reduces to the $D$-topology.
}

\subsubsection{Differences between \kathrin{diffeological spaces} and manifolds}

Manifolds can be generalized in many ways. In \cite{Stacey}, a summary and comparison of possibilities to generalize smooth manifolds are given. 
One generalization is a diffeological space on which we concentrate in this section.
In the following, the main differences between manifolds and diffeological spaces are figured out \kathrin{and formulated in  \begin{rrev}Theorem~\ref{theorem_diffman}\end{rrev}}. 
For simplicity, we concentrate on finite-dimensional manifolds. However, it has to be mentioned that infinite-dimensional manifolds can also be understood as diffeological spaces. This follows, e.g., from \cite[Corollary~3.14]{KrieglMichor} or \cite{Losik}.

Given a smooth manifold there is a natural diffeology on this manifold consisting of all parametrizations which are smooth in the classical sense.
This yields the following definition.

\begin{definition}[Diffeological space associated with a manifold]
	Let $M$ be a finite-dimensional (not necessarily Hausdorff or paracompact) smooth manifold. The diffeological space associated with $M$ is defined as $(M,D_M)$, where the diffeology $D_M$ consists precisely of the parametrizations of $M$ which are smooth in the classical sense.
\end{definition}

\begin{remark}
	\label{remark_diff}
	If $M,N$ denote finite-dimensional manifolds, then $f\colon M\to N$ is smooth in the classical sense if and only if it is a smooth map between the associated diffeological spaces $(M,D_M)\to(N,D_N)$.
\end{remark}

In order to characterize the diffeological spaces which arise from manifolds, we need the concept of \emph{smooth points}.

\begin{definition}[Smooth point]
	\label{definiton_smoothpoints}
	Let $(X,D_X)$ be a diffeological space. A point $x\in X$ is called smooth if there exists a subset $U\subset X$  \kathrin{which is open with respect to the topology of $X$ and} contains $x$ such that \kathrin{$(U,D_U)$} is diffeomorphic to an open subset of $\mathbb{R}^n$, \kathrin{where $D_U$ denotes the subset diffeology}.
\end{definition}

\noindent
The concept of smooth points is quite simple. Let us consider the coordinate axes, e.g., in $\mathbb{R}^2$. All points of the two axis with exception of the origin are smooth points.

Now, we are able to formulate the following \kathrin{main} theorem: 

\begin{theorem}
	\label{theorem_diffman}
	A diffeological space $(X,D_X)$ is associated with a (not necessarily paracompact or Hausdorff) smooth manifold if and only if each of its points is smooth.
\end{theorem}

\emph{Proof.}
We have to show the following statements:
\begin{itemize}
	\item[(i)] Given a smooth manifold $M$, then each point of the associated diffeological space $(M,D_M)$ is smooth.
	\item[(ii)] Given a diffeological space $(X,D_X)$ for which all points are smooth, then it is associated with a smooth manifold $M$.
\end{itemize}

\textbf{To (i):}
Let $M$ be a smooth manifold and $x\in M$ an arbitrary point. Then there exists an open neighbourhood $U\subset M$ of $x$ which is diffeomorphic to an open subset $O\subset \mathbb{R}^n$. Let $f\colon U\to O$ be a diffeomorphism. This diffeomorphism is a diffeomphism of the associated diffeological spaces $(U,D_U)$ and $(O,D_O)$. Thus, $x\in M$ is a smooth point. Since $x\in M$ is an arbitrary point, each point of $(M,D_M)$ is smooth.

\textbf{To (ii):}
Let $(X,D_X)$ be a diffeological space for which all points are smooth. Then there exist an open cover $X=\bigcup_{i\in I} U_i$ and diffeomorphisms $f_i\colon U_i\to O_i $ onto open subsets $O_i\subset\mathbb{R}^{n}$. The map $f_j\circ f_i^{-1}\,\rule[-2mm]{.1mm}{4mm}_{\hspace{.6mm}f_i(U_i\cap U_j)}$ is smooth (in the diffeological sense) for all $i,j\in I$. Due to Remark~\ref{remark_diff}, the map $f_j\circ f_i^{-1}\,\rule[-2mm]{.1mm}{4mm}_{\hspace{.6mm}f_i(U_i\cap U_j)}$ is smooth in the classical sense for all $i,j\in I$. Thus, $\{(U_i,f_i)\}_{i\in I}$ defines a smooth atlas and a manifold structure on $X$ is defined. Let $D$ be the associated diffeology. 
A similar argument as above shows that the diffeology $D$ agrees with the original one $D_X$.
\qed

\smallskip

\noindent
This theorem clarifies the difference between manifolds and diffeological spaces. 
Roughly speaking, a manifold of dimension $n$ is getting by glueing together open subsets of $\mathbb{R}^n$ via diffeomorphisms. In contrast, a diffeological space is formed by glueing together open subsets of $\mathbb{R}^n$ with the difference that the glueing maps are not necessarily diffeomorphisms and that $n$ can vary.
However, note that manifolds deal with charts and diffeological spaces deal with plots.
A system of local coordinates, i.e., a diffeomorphism $p\colon U\to U'$ with $U\subset \mathbb{R}^n$ open and $U'\subset X$ open, can be viewed as a very special kind of plot $U\to X$ which induces an induction on the corresponding diffeological spaces.

\begin{remark}
	Note that we consider smooth manifolds which do not necessary have to be Hausdorff or paracompact. If we understand a manifold as Hausdorff and paracompact, then the diffeological space $(X,D_X)$ in Theorem~\ref{theorem_diffman} has to be Hausdorff and paracompact.
	In this case, we need the concept of open sets in diffeological spaces.
	Whether a set is open depends on the topology under consideration.
	In the case of diffeological spaces, openness depends on the $D$-topology.
\end{remark}

\subsection{The diffeological shape space}
\label{subsection_shapespaceB12}

We extend the definition of smooth shapes, which are elements of the shape space $B_e$, to \emph{shapes of class $H^{1/2}$}. In the following, it is clarified what we mean by $H^{1/2}$-shapes. 
We would like to recall that a shape in the sense of the shape space $B_e$ is given by the image of an embedding from the unit sphere $S^{d-1}$ into the Euclidean space $\mathbb{R}^d$. In view of our generalization, it has technical advantages to consider so-called \emph{Lipschitz shapes} which are defined as follows.

\begin{definition}[Lipschitz shape]
	A $(d-1)$-dimensional Lipschitz shape $\Gamma_0$ is defined as the boundary $\Gamma_0=\partial\mathcal{X}_0$ of a compact Lipschitz domain $\mathcal{X}_0\subset\mathbb{R}^d$ with $\mathcal{X}_0\neq\emptyset$. 
	The set $\mathcal{X}_0$ is called a Lipschitz set.
\end{definition}

\noindent
Example of Lipschitz shapes are illustrated in Figure \ref{figure_priorshapes}. In contrast, Figure \ref{figure_nonpriorshapes} shows examples of shapes which are non-Lipschitz shapes.

General shapes---in our novel terminology---arise from $H^1$-deformations of a Lipschitz set $\mathcal{X}_0$. These $H^1$-deformations, evaluated at a Lipschitz shape $\Gamma_0$, give deformed shapes $\Gamma$ if the deformations are injective and continuous. These shapes are called of class $H^{1/2}$ and proposed firstly in \cite{schulz2015Steklov}. The following definitions differ from \cite{schulz2015Steklov}. This is because of our aim to define the space of $H^{1/2}$-shapes as diffeological space which is suitable for the formulation of optimization techniques and its applications. 

\begin{definition}[Shape space $\mathcal{B}^{1/2}$]
	\label{definition_B12}
	Let $\Gamma_0\subset \mathbb{R}^d$ be a $(d-1)$-dimensional Lipschitz shape. The space of all $(d-1)$-dimensional $H^{1/2}$-shapes is given by
	\begin{equation}\label{shape_mainifold}
	{\cal B}^{1/2}(\Gamma_0,\mathbb{R}^d):=
	{\cal H}^{1/2}(\Gamma_0,\mathbb{R}^d)\big\slash \sim \ ,
	\end{equation}
	where 
	\begin{equation}
	\label{force-ball}
	\begin{split}
	& {\cal H}^{1/2}(\Gamma_0,\mathbb{R}^d)\\ &:=
	\{w\colon  w\in H^{1/2}(\Gamma_0, \mathbb{R}^d) \text{ injective, continuous; } w(\Gamma_0) \text{ Lipschitz shape} \}
	\end{split}
	\end{equation}
	and the equivalence relation $\sim$ is given by
	\begin{equation}
	\label{equiv_rel}
	w_1\sim w_2 \Leftrightarrow w_1(\Gamma_0)=w_2(\Gamma_0), \text{ where } w_1,w_2\in {\cal H}^{1/2}(\Gamma_0,\mathbb{R}^d).
	\end{equation}
\end{definition}

\begin{figure}
	\vspace*{-.8cm}
	\begin{minipage}{.3\textwidth}
		\vspace{1.3cm}
		\begin{center}
			\includegraphics[height=2.3cm]{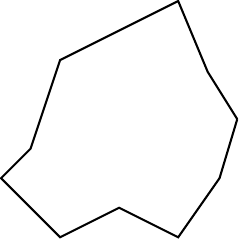}
		\end{center}
		\centering One-dimensional case.
	\end{minipage}
	\hspace{.3cm}
	\begin{minipage}{.8\textwidth}
		\includegraphics[height=3.2cm]{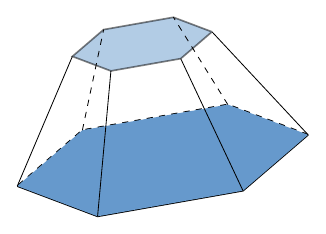}
		\hspace{-8mm}\includegraphics[height=3.9cm]{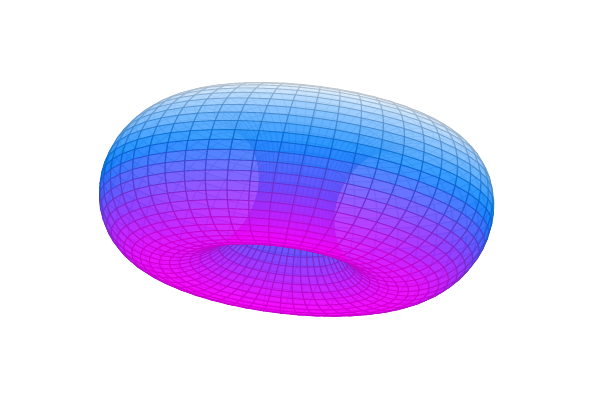}
		\begin{center}
			\vspace{-.2cm}
			Two-dimensional Lipschitz shapes.
		\end{center}
	\end{minipage}
	\caption{Examples of elements of $\mathcal{B}^{1/2}$ and, thus, Lipschitz shapes in the one- and two-dimensional case. The illustrated shapes are not elements of $B_e$.}
	\label{figure_priorshapes}
\end{figure}

The set ${\cal H}^{1/2}(\Gamma_0,\mathbb{R}^d)$ is obviously a subset of the Sobolev-Slobodeckij space $H^{1/2}(\Gamma_0,\mathbb{R}^d)$, which is well-known as a Banach space (cf.~\cite[Chapter~3]{Mclean}).
Banach spaces are manifolds and, thus, we can view $H^{1/2}(\Gamma_0,\mathbb{R}^d)$ with the corresponding diffeology.
This encourages the following theorem which provides the space of $H^{1/2}$-shapes with a diffeological structure. 
\begin{rrev}
	The next theorem is the third (and, thus, last) main theorem of this paper.
\end{rrev}

\begin{theorem}
	\label{Theorem:DiffStructure}
	The set ${\cal H}^{1/2}(\Gamma_0,\mathbb{R}^d)$ and the space ${\cal B}^{1/2}(\Gamma_0,\mathbb{R}^d)$ carry unique diffeologies such that the inclusion map $\iota_{{\cal H}^{1/2}(\Gamma_0,\mathbb{R}^d)}\colon {\cal H}^{1/2}(\Gamma_0,\mathbb{R}^d) \to H^{1/2}(\Gamma_0,\mathbb{R}^d)$ is an induction and such that the canonical projection $\pi\colon {\cal H}^{1/2}(\Gamma_0,\mathbb{R}^d) \to {\cal B}^{1/2}(\Gamma_0,\mathbb{R}^d)$ is a subduction.
	
\end{theorem}

\emph{Proof.}
Let $D_{H^{1/2}(\Gamma_0,\mathbb{R}^d)}$ be the diffeology on $H^{1/2}(\Gamma_0,\mathbb{R}^d)$.
Due to Theorem and Definition~\ref{def_subsetdiff}, ${\cal H}^{1/2}(\Gamma_0,\mathbb{R}^d)$ carries the subset diffeology $\iota_{{\cal H}^{1/2}(\Gamma_0,\mathbb{R}^d)}^\ast\left(D_{H^{1/2}(\Gamma_0,\mathbb{R}^d)}\right)$.
Then the space ${\cal B}^{1/2}(\Gamma_0,\mathbb{R}^d)$ carries the quotient diffeology 
\begin{equation}
\label{diffeology_B}
D_{{\cal B}^{1/2}(\Gamma_0,\mathbb{R}^d)}:=\pi_\ast\left(\iota_{{\cal H}^{1/2}(\Gamma_0,\mathbb{R}^d)}^\ast\left(D_{H^{1/2}(\Gamma_0,\mathbb{R}^d)}\right)\right)
\end{equation}
due to Theorem and Definition~\ref{def_quotdiff}. 
\qed

\smallskip

So far, we have defined the space of $H^{1/2}$-shapes and showed that it is a diffeological space.
The appearance of a diffeological space in the context of shape optimization can be seen as a first step or motivation towards the formulation of optimization techniques on diffeological spaces. Note that, so far, there is no theory for shape optimization on diffeological spaces.
Of course, properties of the shape space ${\cal B}^{1/2}\left(\Gamma_0,\mathbb{R}^d\right)$ have to be investigated.
E.g., an important question is \emph{how the tangent space looks like}. 
Tangent spaces and tangent bundles are important in order to state the connection of ${\cal B}^{1/2}\left(\Gamma_0,\mathbb{R}^d\right)$ to shape calculus and in this way to be able to formulate optimization algorithms in ${\cal B}^{1/2}\left(\Gamma_0,\mathbb{R}^d\right)$. 
There are many equivalent ways to define tangent spaces of manifolds, e.g., geometric via velocities of curves, algebraic via derivations or physical via cotangent spaces (cf.~\cite{Kuehnel}). 
Many authors have generalized these concepts to diffeological spaces, e.g., \cite{ChristensenWu,Hector,Iglesias,Souriau}. In \cite{Souriau}, tangent
spaces are defined for diffeological groups by identifying smooth curves using certain states.
Tangent spaces and tangent bundles for many diffeological spaces are given in \cite{Hector}. Here smooth curves and a more intrinsic identification are used. 
However, in \cite{ChristensenWu}, it is pointed out that there are some errors in \cite{Hector}.
In \cite{Iglesias}, the tangent space to a diffeological space at a point is defined as a
subspace of the dual of the space of 1-forms at that point. These are used to define tangent bundles.
In \cite{ChristensenWu}, two approaches to the tangent space of a general diffeological space at a point are studied. The first one is the approach introduced in \cite{Hector} and the second one is an approach which uses smooth derivations on germs of smooth real-valued functions. Basic facts about these tangent spaces are proven, e.g., locality and that the internal tangent space respects finite products.
Note that the tangent space to ${\cal B}^{1/2}\left(\Gamma_0,\mathbb{R}^d\right)$ as diffeological space and related objects which are needed in optimization methods, e.g., retractions and vector transports, cannot be deduced or defined so easily. 
The study of these objects and the formulation of optimization methods on a diffeological space go beyond the scope of this paper and are topics of subsequent work.
Moreover, note that the Riemannian structure $g^S$ on ${\cal B}^{1/2}\left(\Gamma_0,\mathbb{R}^d\right)$ has to be investigated in order to define ${\cal B}^{1/2}\left(\Gamma_0,\mathbb{R}^d\right)$ as a \emph{Riemannian diffeological space}. In general, a diffeological space can be equipped with a Riemannian structure as outlined, e.g., in \cite{Magnot}. 

Besides the tangent spaces, another open question is \emph{which assumptions guarantee that the image of a Lipschitz shape under $w\in H^{1/2}\left(\Gamma_0,\mathbb{R}^d\right)$ is again a Lipschitz shape}. 
Of course, the image of a Lipschitz shape under a continuously differentiable function is again a Lipschitz shape, but the requirement that $w$ is a $\mathcal{C}^1$-function is a too strong. One idea is to require that $w$ has to be a bi-Lipschitz function. Unfortunately, the image of a Lipschitz shape under a bi-Lipschitz function is not necessarily a Lipschitz shape as the example given in \cite[Subsection~4.1]{Lipschitz-shape} shows.
\kathrin{
	Another option is to generalize the concept of Lipschitz domains to non-tangentially accessible (NTA) domains. In order to formulate the definition of these domains, we need the concept of  so-called \emph{Harnack chains}. 
	
	\begin{definition}[$\alpha$-Harnack chain]
		Let $\alpha\geq1$. For a metric space $(\mathbb{R}^d,\text{d})$ and an open set $\Omega\subset \mathbb{R}^d$, a sequence of balls $B_0, \dots,B_k\subset\Omega$ is called an $\alpha$-Harnack chain in $\Omega$ if $B_i\cap B_{i-1}\not= \emptyset $ for all $i=1,\dots,k$ and $$\alpha^{-1} \text{\emph{dist}}(B_i,\partial\Omega)\leq r(B_i)\leq \alpha\, \text{\emph{dist}}(B_i,\partial\Omega),$$ where $\text{\emph{dist}}(B_i,\partial\Omega):=\inf\limits_{x\in B_i, y\in\partial\Omega} \text{\emph{d}}(x,y)$ and $r(B_i)$ is the radius of $B_i$.
	\end{definition}
	
	\begin{definition}[Non-tangentially accessible domain]
		Let $(\mathbb{R}^d,\text{\emph{d}})$ be a metric space. A bounded open set $\Omega$ is called an non-tangentially accessible domain (NTA domain) if the following conditions hold:
		\begin{itemize}
			\item[(i)] There exist $\alpha\geq 1$ such that for all $\eta>0$ and for all $x,y\in\Omega$ such that $\text{\emph{dist}}(x,\partial\Omega)\geq\eta$, $\text{\emph{dist}}(y,\partial\Omega)\geq\eta$ and $\text{\emph{d}}(x,y)\leq C\eta$ for some $C>0$, there exists an $\alpha$-Harnack chain $B_0,\dots,B_k\subset\Omega$ such that $x\in B_0,y\in B_k$ and $k$ depends on $C$ but not on $\eta$.
			\item[(ii)] $\Omega$ satisfies the corkscrew condition, i.e., there exist $r_0>0$ and $\varepsilon>0$ such that for all $r\in (0,r_0)$ and $x\in \partial\Omega$ the sets $B(x,r)\cap\Omega$, $B(x,r)\cap (\mathbb{R}^d\setminus \overline{\Omega})$ contain a ball of radius $\varepsilon r$.
		\end{itemize}
	\end{definition}
	\noindent
	In fact, the image of an non-tangentially accessible (NTA) domain under a global quasiconformal mapping is an NTA domain (cf.~\cite{NTA1}). 
	If we consider boundaries $\Gamma_0$ of NTA domains $\mathcal{X}_0$ instead of Lipschitz domains, the space $\mathcal{H}^{1/2}$ defined in (\ref{force-ball}) changes to
	\begin{equation}
	\label{H}
	\begin{split}
	&{\cal H}^{1/2}(\Gamma_0,\mathbb{R}^d)\\&:=
	\{w\colon  w\in H^{1/2}(\Gamma_0, \mathbb{R}^d) \text{ injective, continuous; }w=\text{tr}\,W \text{ with }\\
	& \hspace{1.2cm} W\in H^1(\mathcal{X}_0, \mathbb{R}^d) \text{ global quasiconformal and } \Gamma_0=\partial\mathcal{X}_0\}.
	\end{split}
	\end{equation}
	The resulting space of   $H^{1/2}$-shapes carries also a diffeological structure if $\Gamma_0$ is the boundary of an NTA domain $\mathcal{X}_0$ due to Theorem~\ref{Theorem:DiffStructure}.
	
	\begin{remark}
		A quasi-conformal mapping of the open $d$-ball $B^{d-1}$ induces a homeomorphism on the boundary for $d=2,3$ (cf.~\cite{mori1957quasi,gehring1962rings}). In \cite{mostow1968quasi}, this result is generalized for higher dimensions. More precisely, it is proven that a quasiconformal mapping of an open ball in $\mathbb{R}^d$ onto itself extends to a homeomorphism of the closed $d$-ball. If we apply these results to our shape space, we get injectivity and continuity in (\ref{H}) for free for $\Gamma_0:=S^{d-1}$. 
	\end{remark}
}

\begin{figure}
	\vspace*{-1.5cm}
	\begin{minipage}{.3\textwidth}
		\vspace{2.3cm}
		\begin{center}
			\includegraphics[height=2cm]{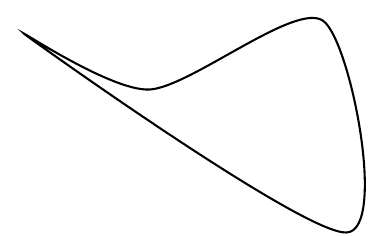}\\[.6cm]
			One-dimensional case
		\end{center}
	\end{minipage}
	\vspace{.4cm}
	\begin{minipage}{.85\textwidth}
		\vspace{.4cm}
		\hspace{-.1cm}\includegraphics[width=.52\textwidth]{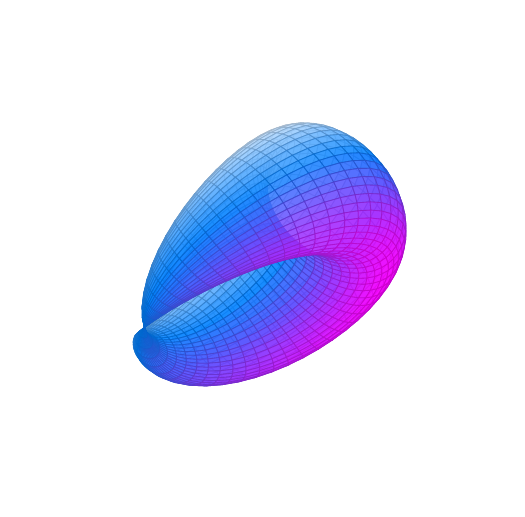}
		\hspace{-2cm}\includegraphics[width=.52\textwidth]{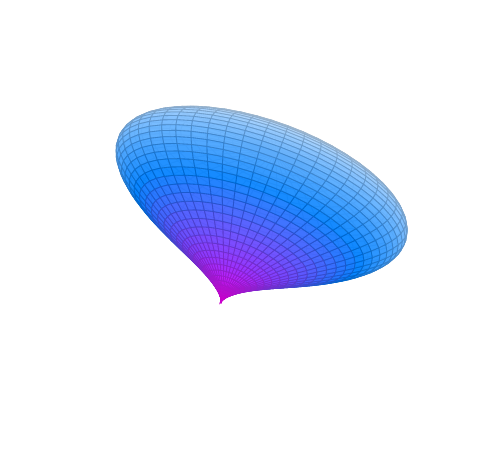}
		\begin{center}
			\vspace*{-1.2cm}Two-dimensional non-Lipschitz shapes.
		\end{center}
	\end{minipage}
	\caption{Examples of non-Lipschitz shapes.}
	\label{figure_nonpriorshapes}
\end{figure}

%%%%%%%%%%%%%%%%%%%%%%%%%%%%%%%%%%%%%%%%%%%%%%%%%%%%%%%%%%%%%%%%%%%%%%%%%%%%%%%%%%%%%%%%%%%%%%%%%

\section{Conclusion}

The differential-geometric structure of the shape space $B_e$ is applied to the theory of shape optimization problems. In particular, a Riemannian shape gradient and a Riemannian shape Hessian with respect to the Sobolev metric $g^1$ is defined. The specification of the Riemannian shape Hessian requires the Riemannian connection, which is given and proven for the \begin{rrev}
	first\end{rrev} Sobolev metric. 
It is outlined that we have to deal with surface formulations of shape derivatives if we consider the \begin{rrev}
	first\end{rrev} Sobolev metric. In order to use the more attractive volume formulations, we consider the Steklov-Poincar\'{e} metrics $g^S$ and state their connection to shape calculus by defining the shape gradient with respect to $g^S$. 
The gradients with respect to both, $g^1$ and $g^S$, and the Riemannian shape Hessian, open the door to formulate optimization algorithms in $B_e$.
\begin{rrev}We formulate the gradient method in $(B_e,g^1)$ and $(B_e,g^S)$ as well as the Newton method in $(B_e,g^1)$.\end{rrev}
\begin{rrev}
	The implementation and investigation of Newton's method in $(B_e,g^1)$ for an explicit example will be touched in future work. 
	In particular, the comparison of Newton's method in $(B_e,g^1)$ and $(B_e,g^S)$ will be investigated in the future.
	Here, a challenging question, which arises, is how an explicit formulation of the Riemannian shape Hessian with respect to the  Steklov-Poincar\'{e} metric looks like. For this, we would generally need to work in fractional order Sobolev spaces and deal with the projected Poincar\'{e}-Steklov operator.
\end{rrev}

Since the shape space $B_e$ limits the application of optimization techniques, we extend the definition of smooth shapes to $H^{1/2}$-shapes and define a novel shape space. It is shown that this space has a diffeological structure.
In this context, we clarify the differences between manifolds and diffeological spaces.
From a theoretical point of view, a diffeological space is very attractive in shape optimization. 
It can be supposed that a diffeological structure suffices for many differential-geometric tools used in shape optimization techniques. 
In particular, objects which are needed in optimization methods, e.g., retractions and vector transports, have to be deduced. Note that these objects cannot be defined so easily and additional work is required to formulate optimization methods on a diffeological space, which remain open for further research and will be touched in subsequent papers.

%%%%%%%%%%%%%%%%%%%%%%%%%%%%%%%%%%%%%%%%%%%%%%%%%%%%%%%%%%%%%%%%%%%%%%%%%%%%%%%%%%%%%%%%%%%%%%%%%

\section*{Acknowledgement}
The author is indebted to Ben Anthes for many helpful comments and discussions about diffeological spaces.
Moreover, the author thanks Leonhard Frerick (Trier University) for discussions about Lipschitz domains.

%%%%%%%%%%%%%%%%%%%%%%%%%%%%%%%%%%%%%%%%%%%%%%%%%%%%%%%%%%%%%%%%%%%%%%%%%%%%%%%%%%%%%%%%%%%%%%%%%

\bibliographystyle{abbrv}
\bibliography{citations.bib}

\end{document}